\documentclass[11pt]{article}
\usepackage{authblk}
\usepackage[toc,page]{appendix}
\usepackage[top=2cm, bottom=2cm, left=2cm, right=2cm]{geometry}

\usepackage{color}
\usepackage{helvet}         
\usepackage{courier}        
\usepackage{type1cm}        

\usepackage{framed} 
\usepackage{tikz}
\usepackage{makeidx}         
\usepackage{graphicx}        
\usepackage{multicol}        
\usepackage[bottom]{footmisc}

\usepackage{amsmath}
\usepackage{amssymb}
\usepackage{bbold}
\usepackage{amsthm}
\usepackage{subcaption}
\usepackage{sidecap}
\usepackage{floatrow}
\usepackage{pdflscape}
\usepackage{comment}
\usepackage[font=small]{caption}
\usepackage{enumitem}
\usepackage{esint}

\usepackage{scalerel}

\newtheorem{theorem}{Theorem}
\newtheorem{corollary}[theorem]{Corollary}
\newtheorem{lemma}[theorem]{Lemma}

\newtheorem{proposition}[theorem]{Proposition}

\theoremstyle{definition}

\theoremstyle{remark}
\newtheorem{remark}[theorem]{Remark}

\numberwithin{theorem}{section}
\numberwithin{figure}{section}
\numberwithin{equation}{section}

\begin{document}

\title{Gaussian free field with logarithmic singularities \\and reduced Green's energy}
\bigskip{}
\author[1]{Mingchang Liu\thanks{liumc\_prob@163.com}}
\author[2]{Titus Lupu\thanks{titus.lupu@sorbonne-universite.com}}
\author[3]{Hao Wu\thanks{hao.wu.proba@gmail.com. ORCID:0000-0003-4265-7417.}}
\affil[1]{Captial Normal University, China}
\affil[2]{Sorbonne Universit\'e and Universit\'e Paris Cit\'e, CNRS, France}
\affil[3]{Tsinghua University, China}
\date{}

%
%


\global\long\def\U{\mathbb{U}}
\global\long\def\T{\mathbb{T}}
\global\long\def\HH{\mathbb{H}}
\global\long\def\R{\mathbb{R}}
\global\long\def\C{\mathbb{C}}
\global\long\def\N{\mathbb{N}}
\global\long\def\Z{\mathbb{Z}}
\global\long\def\E{\mathbb{E}}
\global\long\def\PP{\mathbb{P}}
\global\long\def\QQ{\mathbb{Q}}
\global\long\def\A{\mathbb{A}}
\global\long\def\one{\mathbb{1}}

\global\long\def\Green{\mathrm{G}}
\global\long\def\CR{\mathrm{CR}}
\global\long\def\ST{\mathrm{ST}}
\global\long\def\SF{\mathrm{SF}}
\global\long\def\cov{\mathrm{cov}}
\global\long\def\dist{\mathrm{dist}}
\global\long\def\SLE{\mathrm{SLE}}
\global\long\def\hSLE{\mathrm{hSLE}}
\global\long\def\CLE{\mathrm{CLE}}
\global\long\def\GFF{\mathrm{GFF}}
\global\long\def\inte{\mathrm{int}}
\global\long\def\ext{\mathrm{ext}}
\global\long\def\inrad{\mathrm{inrad}}
\global\long\def\outrad{\mathrm{outrad}}
\global\long\def\dimH{\mathrm{dim}}
\global\long\def\capa{\mathrm{cap}}
\global\long\def\diam{\mathrm{diam}}
\global\long\def\sign{\mathrm{sgn}}
\global\long\def\cat{\mathrm{Cat}}
\global\long\def\cst{\mathrm{C}}
\global\long\def\free{\mathrm{free}}
\global\long\def\hF{{}_2\mathrm{F}_1}
\global\long\def\simple{\mathrm{simple}}
\global\long\def\even{\mathrm{even}}
\global\long\def\odd{\mathrm{odd}}
\global\long\def\st{\mathrm{ST}}
\global\long\def\usf{\mathrm{USF}}
\global\long\def\Leb{\mathrm{Leb}}
\global\long\def\LP{\mathrm{LP}}
\global\long\def\Poisson{\mathrm{P}}
\global\long\def\const{\mathrm{C}}

\global\long\def\LA{\mathcal{A}}
\global\long\def\LB{\mathcal{B}}
\global\long\def\LC{\mathcal{C}}
\global\long\def\LD{\mathcal{D}}
\global\long\def\LF{\mathcal{F}}
\global\long\def\LK{\mathcal{K}}
\global\long\def\LE{\mathcal{E}}
\global\long\def\LG{\mathcal{G}}
\global\long\def\LI{\mathcal{I}}
\global\long\def\LJ{\mathcal{J}}
\global\long\def\LL{\mathcal{L}}
\global\long\def\LM{\mathcal{M}}
\global\long\def\LN{\mathcal{N}}
\global\long\def\OO{\mathcal{O}}
\global\long\def\LQ{\mathcal{Q}}
\global\long\def\LR{\mathcal{R}}
\global\long\def\LT{\mathcal{T}}
\global\long\def\LS{\mathcal{S}}
\global\long\def\LU{\mathcal{U}}
\global\long\def\LV{\mathcal{V}}
\global\long\def\LW{\mathcal{W}}
\global\long\def\LX{\mathcal{X}}
\global\long\def\LY{\mathcal{Y}}
\global\long\def\PartF{\mathcal{Z}}
\global\long\def\LH{\mathcal{H}}
\global\long\def\LJ{\mathcal{J}}

\global\long\def\bm{\mathfrak{m}}
\global\long\def\blm{\mathfrak{m}^{\mathrm{loop}}}

\global\long\def\LZ{\mathcal{Z}}
\global\long\def\LZgff{\mathcal{Z}_{\mathrm{GFF}}}
\global\long\def\LZgffr{\mathcal{Z}_{\mathrm{GFF}}^{(\mathfrak{r})}}

\global\long\def\coulomb{\LH}
\global\long\def\auxcoulomb{\hat{\coulomb}}
\global\long\def\coulombGas{\LF}
\global\long\def\coulombnew{\LK}
\global\long\def\coulombLine{\LG}
\global\long\def\kfunc{p}

\global\long\def\eps{\epsilon}
\global\long\def\ov{\overline}
\global\long\def\QQrp{\QQ_{\alpha; \bs{s}}^{(p)}}

\global\long\def\bn{\mathbf{n}}
\global\long\def\MR{MR}
\global\long\def\cond{\,|\,}
\global\long\def\la{\langle}
\global\long\def\ra{\rangle}
\global\long\def\tree{\Upsilon}
\global\long\def\prob{\mathbb{P}}
\global\long\def\hm{\mathrm{Hm}}
%

\global\long\def\Im{\operatorname{Im}}
\global\long\def\Re{\operatorname{Re}}

\global\long\def\ud{\mathrm{d}}
\global\long\def\pder#1{\frac{\partial}{\partial#1}}
\global\long\def\pdder#1{\frac{\partial^{2}}{\partial#1^{2}}}
\global\long\def\pddder#1{\frac{\partial^{3}}{\partial#1^{3}}}
\global\long\def\der#1{\frac{\ud}{\ud#1}}

\global\long\def\bZnn{\mathbb{Z}_{\geq 0}}
\global\long\def\bZpos{\mathbb{Z}_{> 0}}
\global\long\def\bZneg{\mathbb{Z}_{< 0}}

\global\long\def\Vfunc{\LG}
\global\long\def\gfunc{g^{(\rr)}}
\global\long\def\hfunc{h^{(\rr)}}

\global\long\def\SimplexInt{\rho}
\global\long\def\CubeInt{\widetilde{\rho}}

\global\long\def\ii{\mathfrak{i}}
\global\long\def\ee{\mathrm{e}}
\global\long\def\rr{\mathfrak{r}}
\global\long\def\chamber{\mathfrak{X}}
\global\long\def\Wchamber{\mathfrak{W}}

\global\long\def\SimplexIntKappa8{\SimplexInt}

\global\long\def\nested{\boldsymbol{\underline{\Cap}}}
\global\long\def\unnested{\boldsymbol{\underline{\cap\cap}}}
\global\long\def\unnested{\boldsymbol{\underline{\cap\cap}}}

\global\long\def\acycle{\vartheta}
\global\long\def\bcycle{\tilde{\acycle}}

\global\long\def\metric{\mathrm{dist}}

\global\long\def\adj#1{\mathrm{adj}(#1)}

\global\long\def\bs{\boldsymbol}

\global\long\def\edge#1#2{\langle #1,#2 \rangle}
\global\long\def\graph{G}

\newcommand{\conn}{\varsigma}
\newcommand{\realacycle}{\smash{\mathring{\acycle}}}
\newcommand{\realpt}{\smash{\mathring{x}}}
\newcommand{\corrind}{\LC}
\newcommand{\bssymb}{\pi}
\newcommand{\PRCM}{\mu}
\newcommand{\coeff}{p}
\newcommand{\MainConst}{C}

\global\long\def\removeLink{/}

\global\long\def\domainofdef{\mathfrak{U}}
\global\long\def\Test_space{C_c^\infty}
\global\long\def\Distr_space{(\Test_space)^*}

\global\long\def\bs{\boldsymbol}
\global\long\def\cst{\mathrm{C}}

\newcommand{\red}{\textcolor{red}}
\newcommand{\blue}{\textcolor{blue}}
\newcommand{\green}{\textcolor{green}}
\newcommand{\magenta}{\textcolor{magenta}}

\newcommand{\coulombGasH}{\mathcal{H}}
\newcommand{\secondbeta}{\intloop}

\newcommand{\cev}[1]{\reflectbox{\ensuremath{\vec{\reflectbox{\ensuremath{#1}}}}}}

\global\long\def\anticonf{\zeta}
\global\long\def\intloop{\varrho}
\global\long\def\Gloop{\smash{\mathring{\intloop}}}

\global\long\def\SLEmeasure{\mathrm{P}}
\global\long\def\SLEmeasureEx{\mathrm{E}}

\global\long\def\fugacity{\nu}
\global\long\def\meanderMat{\mathcal{M}}
\global\long\def\LM{\mathcal{M}}
\global\long\def\meanderMatrix{\meanderMat_{\fugacity}}
\global\long\def\meanderMatrixPrime{\meanderMat_{\fugacity(\kappa')}}
\global\long\def\meanderRenorm{\widehat{\mathcal{M}}}

\global\long\def\PartFRenorm{\widehat{\PartF}}
\global\long\def\coulombGasRenorm{\widehat{\coulombGas}}

\global\long\def\hexa{\scalebox{1.3}{\hexagon}}

\global\long\def\np{p}

\global\long\def\FKdual{\mathcal{L}}

\global\long\def\fixedindex{\flat}

\global\long\def\covmap{\phi}

\global\long\def\mslitdriv{\omega}
\newcommand{\LZradial}{\mathcal{Z}_{\mathrm{rad}}}
\newcommand{\Pradial}{\mathsf{P}_{\mathrm{rad}}}

\maketitle
\vspace{-1cm}
\begin{center}
\begin{minipage}{0.95\textwidth}
\abstract{We analyze the level lines of the Gaussian Free Field with logarithmic singularities in polygonal domains. We focus on the (none-zero) singular case and compare it to the singular-free setting.
The law of the level lines for the field in singular case is absolutely continuous with respect to the law for the singular-free case, with a Radon–Nikodym derivative related to the reduced Green's energy which is introduced by F. Viklund and D. W. Nystr\"om~[VN24]. 
These level lines admit a Brownian motion representation. 
In the singular-free case, such representation involves Poisson point processes of Brownian loops and excursions; while for the singular case, it additionally includes Poisson point processes of Brownian bubbles and bridges. 
The special cases of one or two singularities yield moment-generating functions of the conformal radius, which can be interpreted as one-point and two-point correlation functions of vertex operators in conformal field theory at the free-boson point. 
While the Radon–Nikodym derivative is bounded away from the singularities, the behavior of the level lines near the singularities is markedly different from the singular-free case, with different asymptotic probabilities for all level lines to remain close to the singularity.}

\end{minipage}
\end{center}

\tableofcontents
\section{Introduction}
The Gaussian free field (GFF) is a fundamental random surface model in probability theory and mathematical physics. It is the high-dimensional analogue of Brownian motion, where the time parameter is replaced by a spatial domain~\cite{SheffieldGFFMath}. 
In the context of conformal field theory (CFT), it is the building block for constructing vertex operators. 
Although the GFF is a distribution and not a function, one can still make sense of its (zero) level lines which separate the positive and negative regions of the field~\cite{DubedatSLEFreeField, SchrammSheffieldDiscreteGFF, SchrammSheffieldContinuumGFF}. The properties of multiple level lines of GFF in polygons are widely investigated~\cite{WangWuLevellinesGFFI, AruSepulvedaWernerBTLStwoDGFF, PeltolaWuGlobalMultipleSLEs}. In this article, we focus on level lines of GFF with logarithmic singularities.

\paragraph*{GFF with logarithmic singularities.}
For $p\ge 1$, we say that $(\Omega; x_1, \ldots, x_p)$ is $p$-polygon if $\Omega\subsetneq\C$ is simply connected and $x_1, \ldots, x_p\in\partial\Omega$ are distinct points lying counterclockwise along the boundary. We assume that $\partial\Omega$ is locally connected and the marked boundary points $x_1, \ldots, x_p$ lie on sufficiently regular boundary segments, e.g. $C^{1+\eps}$ for some $\eps>0$. We denote by $(x_1x_2)$ the boundary arc of $\partial\Omega$ from $x_1$ to $x_2$ in counterclockwise order. 

Fix $N\ge 1$ and $2N$-polygon $(\Omega; \bs{x})$ with $\bs{x}=(x_1, \ldots, x_{2N})$. 
Suppose $\Gamma_0(\cdot)=\Gamma_0(\Omega; \bs{x}; \cdot)$ is the bounded harmonic function in $\Omega$ with the boundary data\footnote{with the convention that $x_{2N+1}=x_1$.} 
\begin{equation}\label{eqn::multiplelog_polygon_boundarydata}
\pi \text{ on }\partial^+\Omega=\cup_{j=1}^N(x_{2j-1}x_{2j}), \qquad 0 \text{ on }\partial^-\Omega=\cup_{j=1}^N(x_{2j}x_{2j+1}).
\end{equation}
Suppose $\Gamma$ is GFF in $(\Omega; \bs{x})$ with alternating boundary data~\eqref{eqn::multiplelog_polygon_boundarydata}. 
Fix $M\ge 1$ and $\bs{\gamma}=(\gamma_1, \ldots, \gamma_M)\in\R^M$ and $\bs{u}=(u_1, \ldots, u_M)$ where $u_1, \ldots, u_M\in \Omega$ are distinct points. 
Consider the level lines of the field
\begin{equation}\label{eqn::multiplelog_GFF_polygon}
\Gamma+\sum_{\ell=1}^M \gamma_{\ell}\Green(\Omega; u_{\ell}, \cdot)
\end{equation}
where $\Green(\Omega; \cdot, \cdot)$ is Green's function on $\Omega$ (see~\eqref{eqn::Green_H}-\eqref{eqn::Green_inv}). 
We will consider the collection of level lines of the field~\eqref{eqn::multiplelog_GFF_polygon}. 
The collection of level lines of the field when $\bs{\gamma}=0$ was studied in~\cite{PeltolaWuGlobalMultipleSLEs}. 
In this article, we focus on the law of level lines with non-zero $\bs{\gamma}$ and compare them to the case when $\bs{\gamma}=0$. 
Let us summarize our conclusions.
\begin{itemize}
\item The law of the collection of the level lines of the field~\eqref{eqn::multiplelog_GFF_polygon} is absolutely continuous with respect to the field $\Gamma$ and the Radon-Nikodym derivative is related to the reduced Green's energy introduced by F. Viklund and D. W. Nystr\"om~\cite{ViklundNystrom}, see Theorem~\ref{thm::multiplelog_polygon}.
\item The level lines of the field~\eqref{eqn::multiplelog_GFF_polygon} have a Brownian motion representation, see Theorem~\ref{thm::Brownian_represent}.
When $\bs{\gamma}=0$, such representation involves Poisson point processes of Brownian loops and Brownian excursions. For non-zero $\bs{\gamma}$, such representation also includes Poisson point processes of Brownian bubbles and Brownian bridges. 
\item  The special cases $M=1$ and $M=2$ give moment-generating functions of the conformal radius, see Propositions~\ref{prop::onelog_CR_fourier} and~\ref{prop::twopoint_correlation}. These generating functions can be viewed as one-point and two-point correlation functions of vertex operators in CFT at the free-boson point $\kappa=4$.
\item From the first conclusion, the law of the level lines of the field~\eqref{eqn::multiplelog_GFF_polygon} are absolutely continuous with respect to the case when $\bs{\gamma}=0$. 
In particular, when bounded away from the singularity, the Radon-Nikodym derivative is bounded. However, the behavior of the level lines when they are close to the singularity can be very different from the case when $\bs{\gamma}=0$. The probability for all level lines close to the sigularity  exhibits different asymptotic with respect to the case when $\bs{\gamma}=0$, see Proposition~\ref{prop::asymp}. 
\end{itemize}

\subsection{Reduced Green's energy}

\paragraph*{Hele-Shaw flow and reduced Green's energy.}
The Hele-Shaw flow is a mathematical model of the propagation of fluid in a cell, describing the evolution of the fluid interface. 
We refer reader to the book~\cite{GustafssonVasilev} for the classical Hele-Shaw flow and various references. The competitive Hele-Shaw flow was introduced in~\cite{ViklundNystrom}. 
This is a variant of the Hele-Shaw flow where several droplets compete for space while preserving their topology. 
To describe the stationary solution for the competitive Hele-Shaw flow problem, F. Viklund and D. W. Nystr\"om introduced reduced Green's energy in~\cite{ViklundNystrom}: for a domain $D$ such that each component is simply connected, and for $\bs{u}=(u_1, \ldots, u_M)$ such that $u_1, \ldots, u_M\in D$ are distinct points, reduced Green's energy is defined as
\begin{align}\label{eqn::reducedGreenenergy_def}
\LE_{\bs{\gamma}}(D; \bs{u})=\sum_{1\le i<k \le M}\gamma_i\gamma_k\Green(D; u_i, u_k)+\sum_{j=1}^M\frac{\gamma_j^2}{2}\log\CR(D; u_j)
\end{align}
where $\Green(D; u_i, u_k)$ is the Green's function of the component of $D$ containing $u_i$ and $u_k$\footnote{If $u_i, u_k$ belong to distinct components of $D$, we set $\Green(D; u_i, u_k)=0$.}, 
and $\CR(D; u_j)$ is conformal radius of the component of $D$ containing $u_j$ (see~\eqref{eqn::CR_H}-\eqref{eqn::CR_cov}). 

\begin{theorem}\label{thm::multiplelog_polygon}
Fix $N\ge 1$ and $2N$-polygon $(\Omega; \bs{x})$ with $\bs{x}=(x_1, \ldots, x_{2N})$. 
Suppose $\Gamma$ is GFF in $(\Omega; \bs{x})$ with alternating boundary data~\eqref{eqn::multiplelog_polygon_boundarydata}. 
We consider the field~\eqref{eqn::multiplelog_GFF_polygon}. 
Let $\eta^j$ be the level line starting from $x_{2j-1}$ for $1\le j\le N$ and denote $\bs{\eta}=(\eta^1, \ldots, \eta^N)$ and we also write $\bs{\eta}=\cup_{j=1}^N\eta^j$. 
We denote by $\Omega_{\bs{\eta}}^+$ the union of the connected components of $\Omega\setminus\bs{\eta}$ having $\partial^+\Omega$ on the boundary 
and denote by $\Omega_{\bs{\eta}}^-$ the union of the connected components of $\Omega\setminus\bs{\eta}$ having $\partial^-\Omega$ on the boundary. 
Denote by $\PP_{\bs{\gamma}}$ the law of $\bs{\eta}$. 
Then $\PP_{\bs{\gamma}}$ is absolutely continuous with respect to $\PP=\PP_{\bs{0}}$ and the Radon-Nikodym derivative is given by 
\begin{align}\label{eqn::multiplelog_polygon_RN}
\frac{\exp\left(\LE_{\bs{\gamma}}(\Omega\setminus\bs{\eta}; \bs{u})\right)}{\exp\left(\LE_{\bs{\gamma}}(\Omega; \bs{u})\right)}\prod_{\ell=1}^M\frac{\left(\one\{u_{\ell}\in\Omega_{\bs{\eta}}^+\}\exp\left(\pi\gamma_{\ell}\right)+\one\{u_{\ell}\in\Omega_{\bs{\eta}}^-\})\right)}{\exp(\gamma_{\ell}\Gamma_0(u_{\ell}))}
\end{align}
where $\LE_{\bs{\gamma}}$ is reduced Green's energy~\eqref{eqn::reducedGreenenergy_def}. 
\end{theorem}

We will prove Theorem~\ref{thm::multiplelog_polygon} in Section~\ref{sec::GFF_levellines}. The key in the proof is a martingale observable in Lemma~\ref{lem::multiple_log_mart}. 

\subsection{Brownian motion representation}
\label{sec::intro_Brownian_represnt}

The Brownian motion representation of GFF level lines is rooted in the deep connection between the GFF and Brownian paths, formalized through isomorphism theorems~\cite{LeJanMarkovPathsLoopsFields, LupuFreeField}. 
In this framework, level lines of the GFF can be constructed by sampling a
Poisson point process of Brownian loops and a Poisson point process of Brownian excursions, and then taking the envelope of these paths~\cite{AruLupuSepulvedaFPSGFFCVGISO, LupuWuLevellineGFF}. For the GFF with logarithmic singularities in Theorem~\ref{thm::multiplelog_polygon}, the previous Brownian motion representation can be generalized where one needs to add Poisson point processes of Brownian bubbles and Brownian bridges. 

\begin{figure}[ht!]
\begin{subfigure}[t]{0.18\textwidth}
\begin{center}
\includegraphics[width=\textwidth]{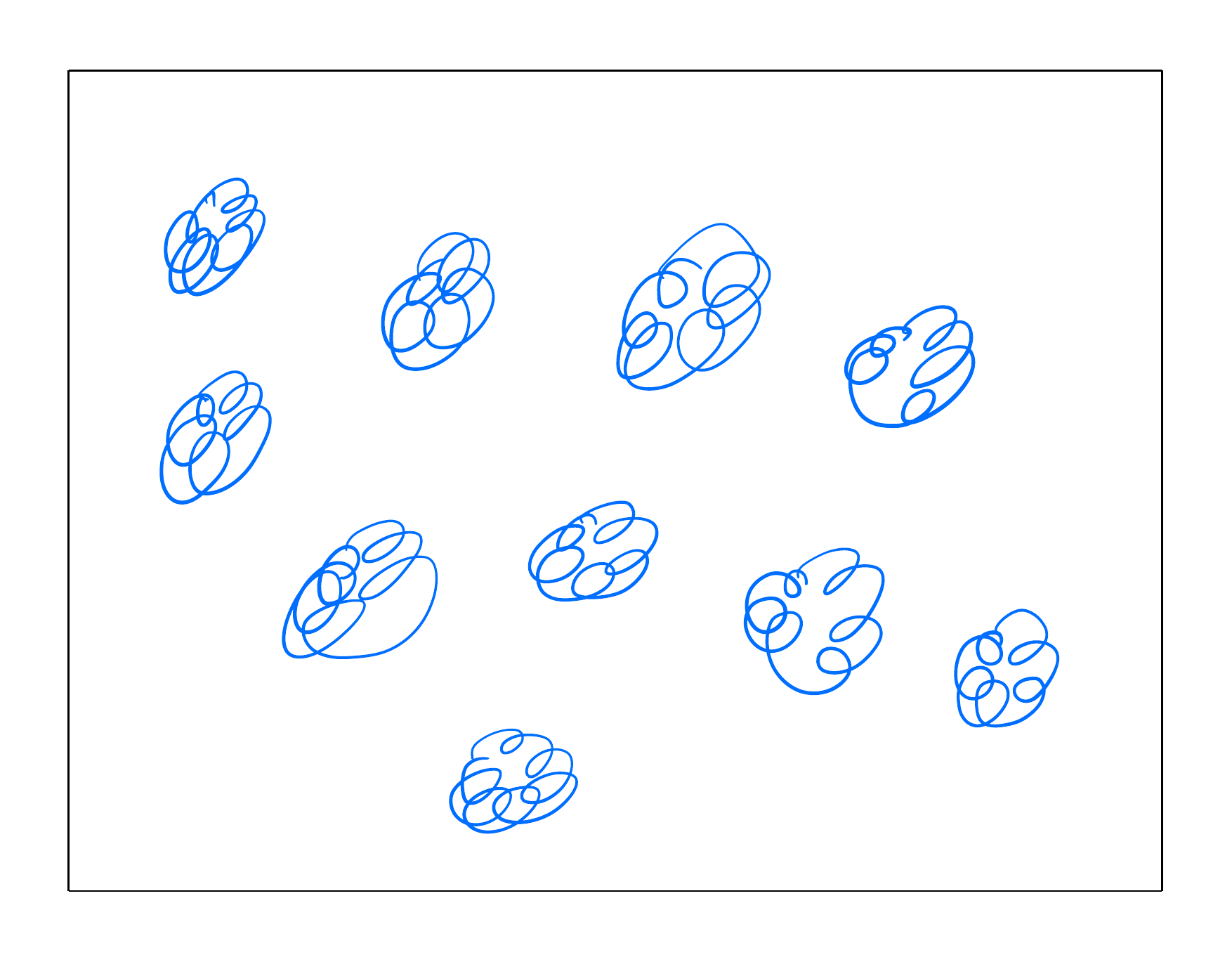}
\end{center}
\caption{}
\end{subfigure}
\begin{subfigure}[t]{0.18\textwidth}
\begin{center}
\includegraphics[width=\textwidth]{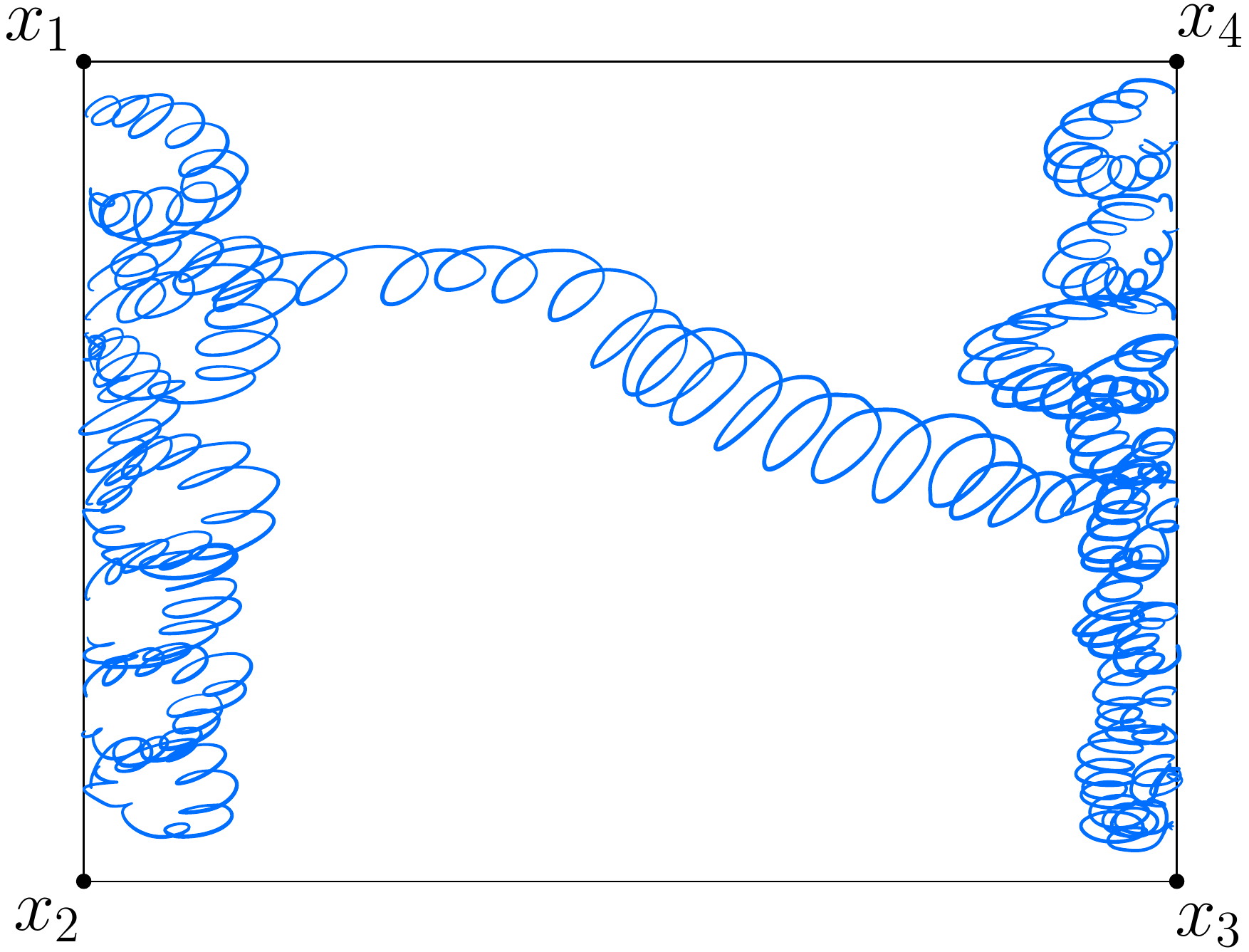}
\end{center}
\caption{}
\end{subfigure}
\begin{subfigure}[t]{0.18\textwidth}
\begin{center}
\includegraphics[width=\textwidth]{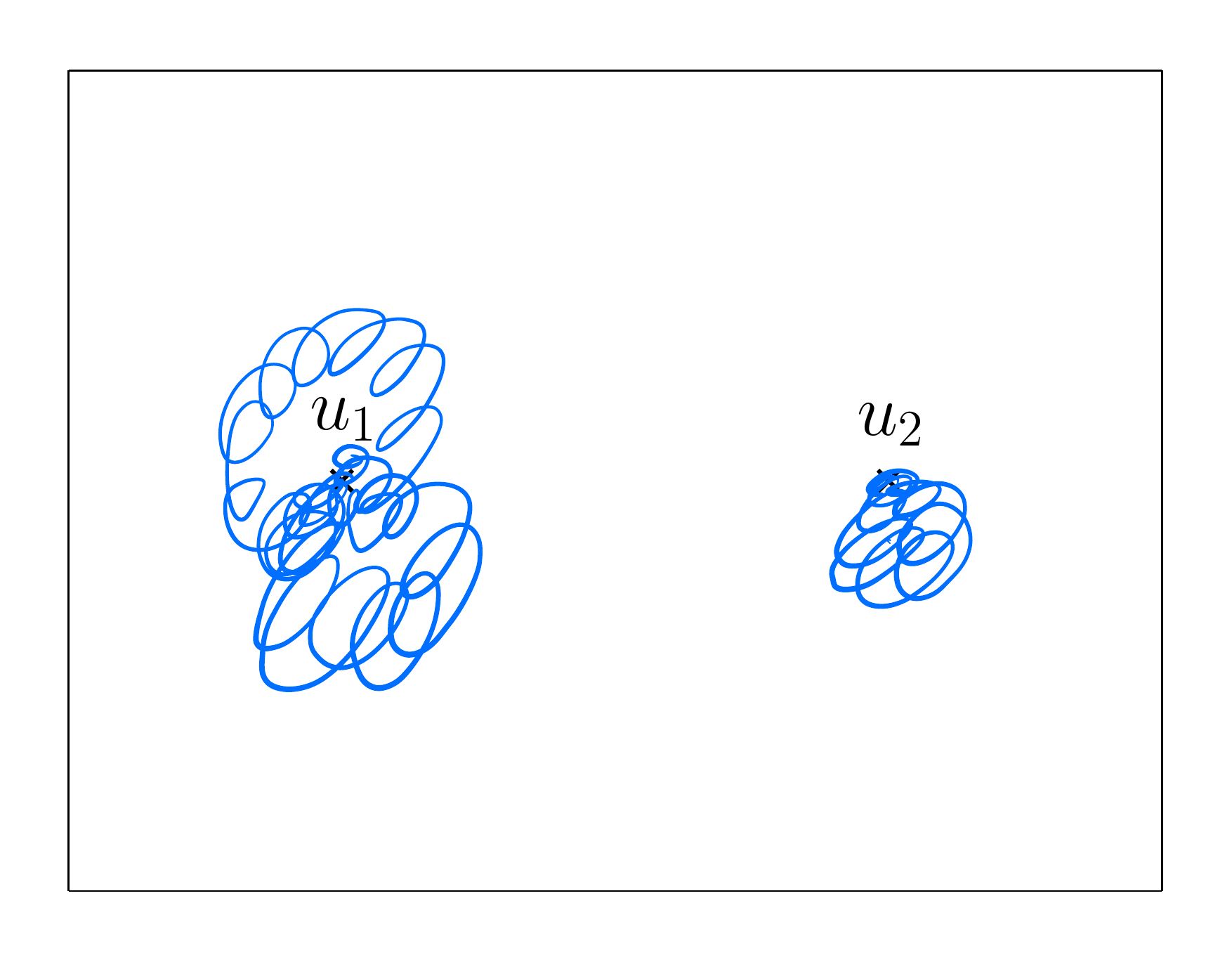}
\end{center}
\caption{}
\end{subfigure}
\begin{subfigure}[t]{0.18\textwidth}
\begin{center}
\includegraphics[width=\textwidth]{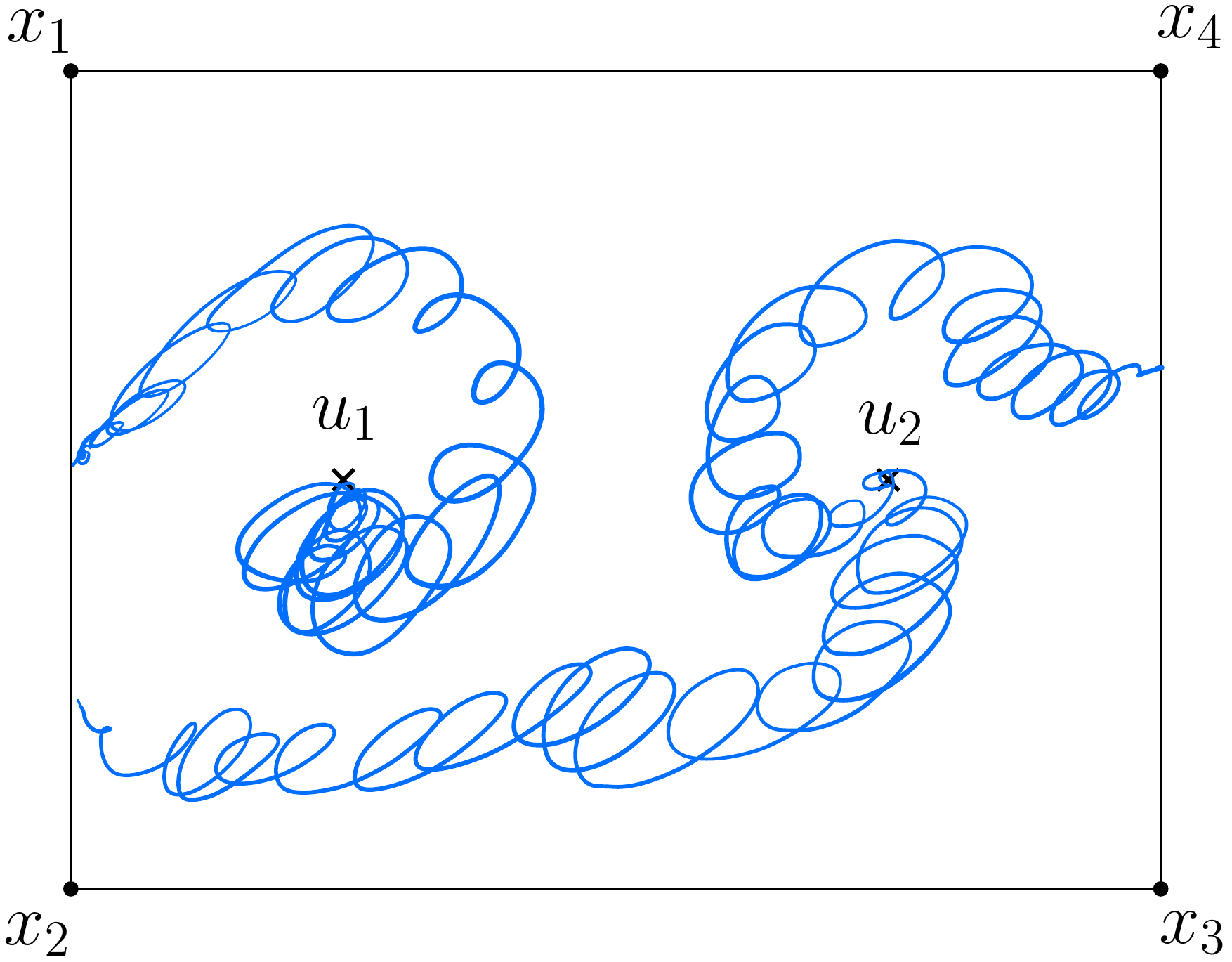}
\end{center}
\caption{}
\end{subfigure}
\begin{subfigure}[t]{0.18\textwidth}
\begin{center}
\includegraphics[width=\textwidth]{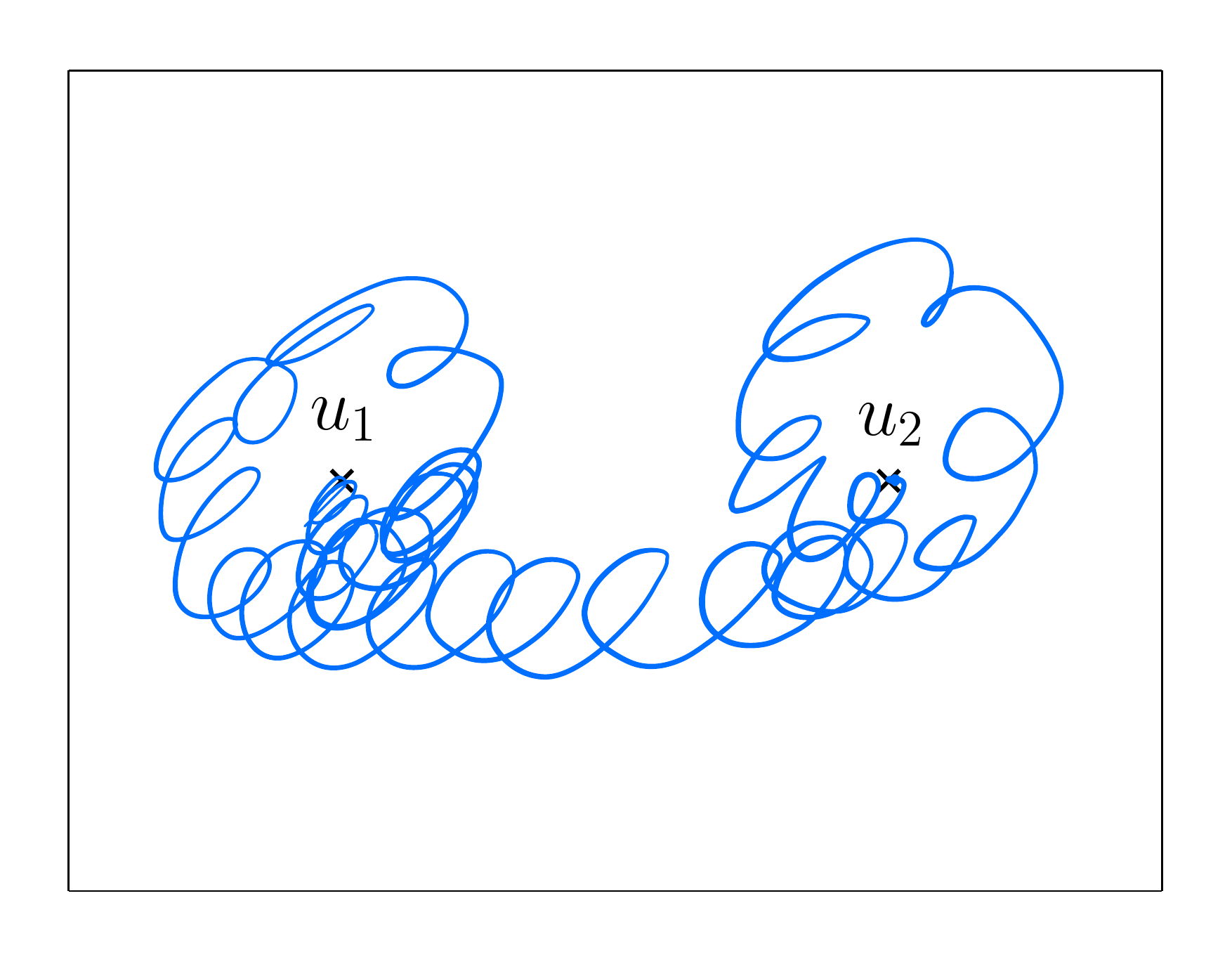}
\end{center}
\caption{}
\end{subfigure}
\caption{\label{fig::PPP} 
Illustration for Poisson point processes in Theorem~\ref{thm::Brownian_represent} for $N=2$ and $M=2$: there are four marked points $x_1, x_2, x_3, x_4$ on the boundary and there are two marked points $u_1, u_2$ in the interior. 
Figure~(a) indicates the loop soup in~(1). 
Figure~(b) indicates the excursion soup in~(2). 
Figure~(c) indicates the bubble soup in~(3). 
Figure~(d) indicates the bridge soup in~(4). 
Figure~(e) indicates the bridge soup in~(5).  
}
\end{figure}

\begin{theorem}\label{thm::Brownian_represent}
Fix $N\ge 1$ and $2N$-polygon $(\Omega; \bs{x})$ with $\bs{x}=(x_1, \ldots, x_{2N})$ and assume the same notation as in~\eqref{eqn::multiplelog_polygon_boundarydata}. 
Fix $M\ge 1$ and $\bs{\gamma}=(\gamma_1, \ldots, \gamma_M)\in\R^M$ and $\bs{u}=(u_1, \ldots, u_M)$ where $u_1, \ldots, u_M\in \Omega$ are distinct points. 
We assume that $\partial\Omega$ is piecewise analytic and that $\gamma_j\ge 0$ for $1\le j\le M$. 
We sample $2+\frac{1}{2}M(M+3)$ independent Poisson point processes, see Figure~\ref{fig::PPP}.
\begin{enumerate}
\item[(1)] Poisson point process of Brownian loops inside $\Omega$ with intensity
$1/2$.
\item[(2)] Poisson point process of Brownian excursions with intensity
\begin{align}
\frac{\pi}{4}\iint_{\partial^+\Omega\times \partial^+\Omega}\ud x\ud y \;\Poisson(\Omega; x, y)\bm^{\#}(\Omega; x, y)[\cdot],
\end{align}
where $\Poisson(\Omega; x, y)$ is boundary Poisson kernel (see~\eqref{eqn::bPoisson_H}-\eqref{eqn::bPoisson_cov}) and $\bm^{\#}(\Omega; x, y)$ is probability measure on Brownian paths in $\Omega$ connecting $x$ and $y$ (see~\eqref{eqn::BM_xy}). 
\item[(3)] For each $u_{\ell}$, Poisson point process of Brownian bubbles with intensity
\begin{align}
\frac{\gamma_{\ell}^2}{4\pi}\bm(\Omega; u_{\ell}, u_{\ell}), 
\end{align}
where $\bm(\Omega; u_{\ell}, u_{\ell})$ is measure on Brownian loops in $\Omega$ that start and end at $u_{\ell}$ (see Section~\ref{subsec::Brownianpath}) which has infinite total mass. 
\item[(4)] For each $u_{\ell}$, Poisson point process of Brownian bridges with intensity
\begin{align}
\frac{\gamma_{\ell}}{2}\int_{\partial^+\Omega}\ud y \;\Poisson(\Omega; u_{\ell}, y)\bm^{\#}(\Omega; u_{\ell}, y)[\cdot], 
\end{align}
where $\Poisson$ is Poisson kernel (see~\eqref{eqn::Poisson_H}-\eqref{eqn::Poisson_cov}) and $\bm^{\#}(\Omega; u_{\ell}, y)$ is probability measure on Brownian paths in $\Omega$ from $u_{\ell}$ to $y$ (see~\eqref{eqn::BM_zy}). 
\item[(5)] For each pair $u_i\neq u_j$, Poisson point process of Brownian bridges with intensity
\begin{align}
\frac{\gamma_i\gamma_j}{2}\frac{\Green(\Omega; u_i, u_j)}{\pi}\bm^{\#}(\Omega; u_i, u_j)[\cdot],
\end{align}
where $\Green(\Omega; u_i, u_j)$ is Green's function (see~\eqref{eqn::Green_H}-\eqref{eqn::Green_inv}) and $\bm^{\#}(\Omega; u_i, u_j)$ is probability measure on Brownian paths in $\Omega$ from $u_i$ to $u_j$ (see~\eqref{eqn::BM_zw}). 
\end{enumerate}
We consider the clusters formed by the collection of all the above Brownian trajectories,
and among these clusters, specifically the clusters connected to the boundary $\partial\Omega$,
which are exactly the clusters containing at least one excursion from (2) or (4).
Let $\bs{\eta}$ be the envelop of these clusters connected to $\partial\Omega$, which consists of
$N$ disjoint curves in $\Omega$ forming a pairing of $\{x_1,\dots ,x_{2N}\}$, see Figure~\ref{fig::envelop}. 
Then $\bs{\eta}$ has the same law as $\PP_{\bs{\gamma}}$ in Theorem~\ref{thm::multiplelog_polygon}. 
\end{theorem}

\begin{figure}[ht!]
\begin{center}
\includegraphics[width=0.2\textwidth]{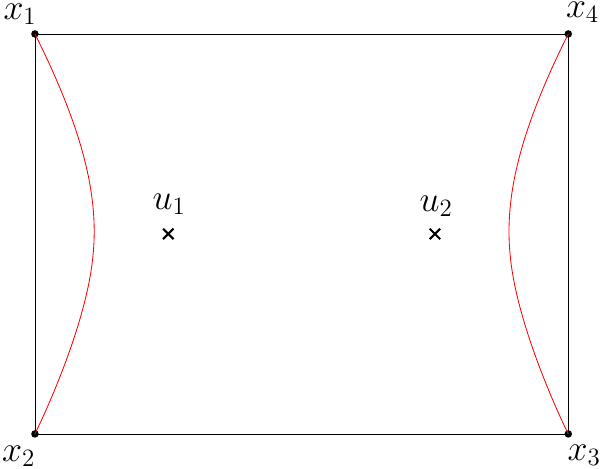}
\includegraphics[width=0.2\textwidth]{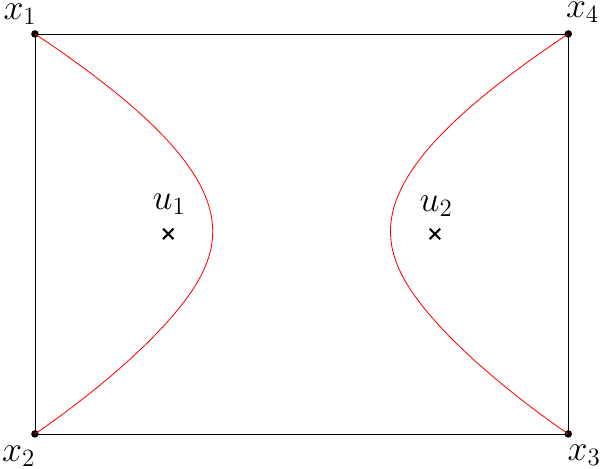}
\includegraphics[width=0.2\textwidth]{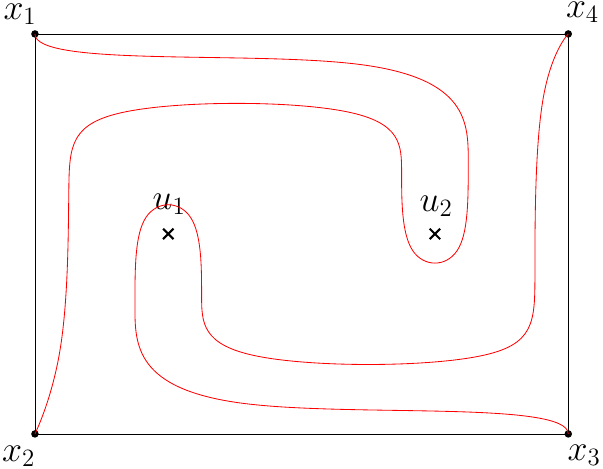}\\
\includegraphics[width=0.2\textwidth]{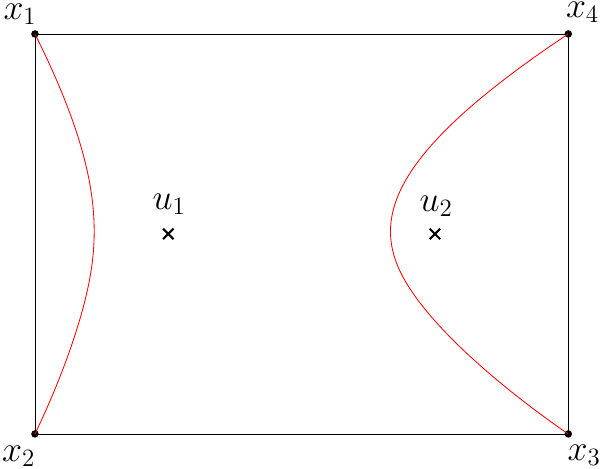}
\includegraphics[width=0.2\textwidth]{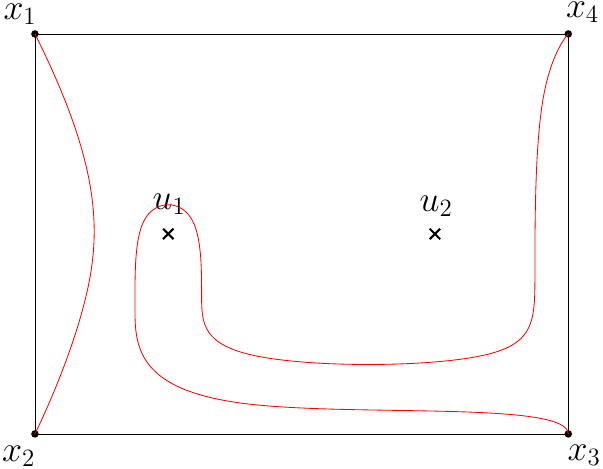}
\includegraphics[width=0.2\textwidth]{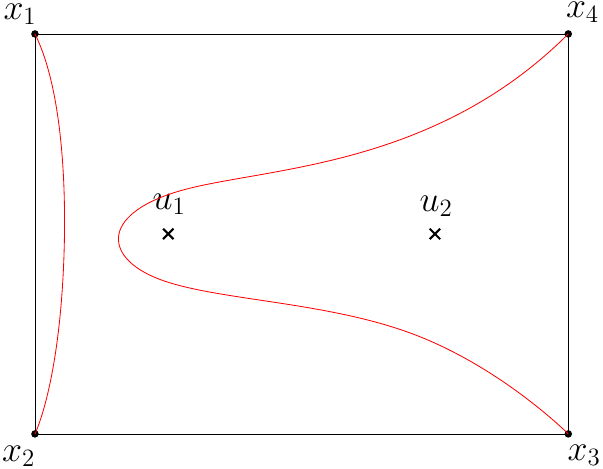}\\
\includegraphics[width=0.2\textwidth]{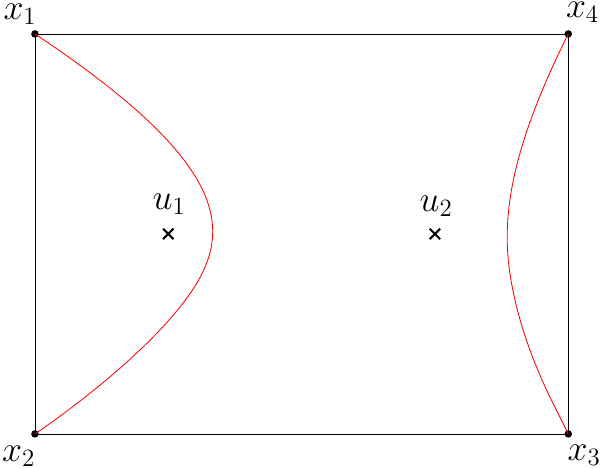}
\includegraphics[width=0.2\textwidth]{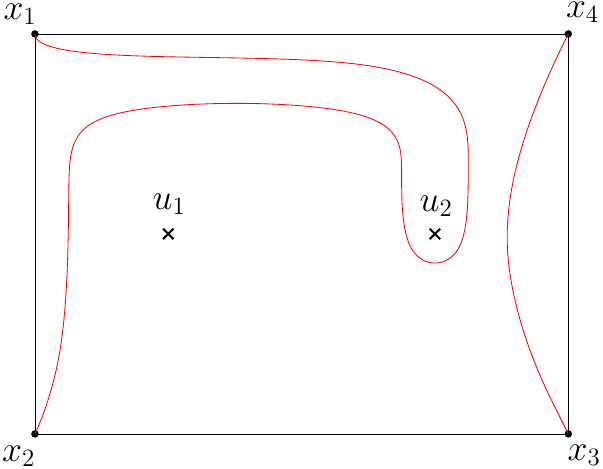}
\includegraphics[width=0.2\textwidth]{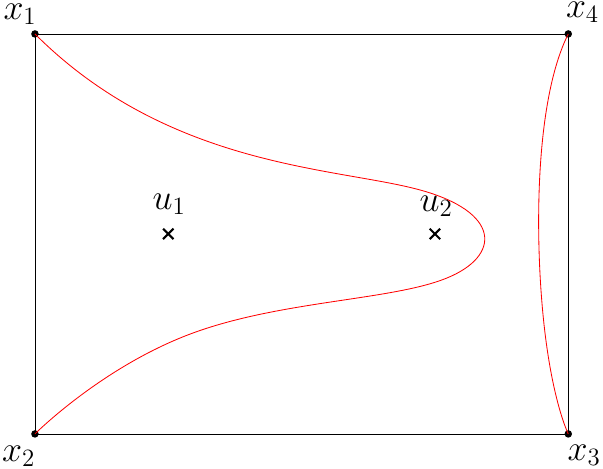}
\end{center}
\caption{\label{fig::envelop} 
Illustration for Poisson point processes in Theorem~\ref{thm::Brownian_represent} for $N=2$ and $M=2$: there are four marked points $x_1, x_2, x_3, x_4$ on the boundary and there are two marked points $u_1, u_2$ in the interior. 
The envelop consists of two disjoint curves connecting $\{x_1, x_2, x_3, x_4\}$ pairwise.
There are two possibilities for the paring and we only present the one when $\{x_1, x_2\}$ are paired and $\{x_3, x_4\}$ are paired. 
}
\end{figure}

Theorem~\ref{thm::Brownian_represent} was proved for $M=0$ in~\cite{AruLupuSepulvedaFPSGFFCVGISO} (see Remark~5.12 and Fig.~4 therein). We prove Theorem~\ref{thm::Brownian_represent} with $M\ge 1$ in Section~\ref{sec::Brownian_represent}. 
The main difference is that the logarithmic singularities lead to unbounded mean value near the marked points, whereas the Brownian representation in~\cite{AruLupuSepulvedaFPSGFFCVGISO} is formulated for bounded mean value. To address this issue, we introduce an $\epsilon$-cutoff around the singularities and apply the Brownian representation in the resulting multiply connected domain $\Omega_\eps$ in Proposition~\ref{prop::BES_truncated}.  We then prove that the corresponding Brownian path ensembles and level lines in $\Omega_\eps$ converge as $\eps\to 0$.
\begin{itemize}
\item
For the convergence of Brownian path ensembles, the main difficulty is that the Brownian paths in Theorem~\ref{thm::Brownian_represent} are rooted at the singularities, whereas those in $\Omega_\eps$ start or end on the boundaries of the holes surrounding the singularities. 
To compare the the two ensembles, we remove parts of the trajectories near the singularities and compare the remaining path ensembles, which yields convergence in total variation. See Lemma~\ref{lem::abs_cont_exc} and Lemma~\ref{lem::conv_envelop_1}.
\item
For the convergence of level lines, the main difficulty is that the GFFs are defined on the varying domains $\Omega_\eps$. Thus, convergence of the fields does not directly imply convergence of the corresponding level lines. We prove the convergence of Radon--Nikodym derivatives between the GFF laws away from the singularities and then use the measurability of the level lines with respect to the fields to obtain convergence in total variation.
\end{itemize}
We emphasize that the Brownian path representation in Theorem~\ref{thm::Brownian_represent} is valid for $\gamma_j\ge 0$ whereas the coupling in Theorem~\ref{thm::multiplelog_polygon} holds for $\gamma_j\in\R$. 

\subsection{Consequences: moment-generating functions}
\label{subsec::intro_moment_generating}
As consequences of Theorem~\ref{thm::multiplelog_polygon}, we derive moment-generating function of the conformal radius in Proposition~\ref{prop::onelog_CR_fourier} and two-point correlation function in Proposition~\ref{prop::twopoint_correlation}. 

\begin{proposition}\label{prop::onelog_CR_fourier}
Assume the same notation as in Theorem~\ref{thm::multiplelog_polygon}  with $M=1$.
Fix $u_1=u\in\Omega$ and $\gamma_1=\gamma\in\R$. For $\xi\in\C$ with $\Re{\xi}<(1+\gamma^2)/2$, we have
\begin{align}\label{eqn::onelog_CR_fourier}
\begin{split}
\E_{\gamma}\left[\left(\frac{\CR(\Omega; u)}{\CR(\Omega\setminus\bs{\eta}; u)}\right)^{\xi}\one\{u\in\Omega_{\bs{\eta}}^+\}\right]
=&\exp(\gamma(\pi-\Gamma_0( u)))\frac{\sin\left(\Gamma_0( u)\sqrt{2\xi-\gamma^2}\right)}{\sin\left(\pi\sqrt{2\xi-\gamma^2}\right)}. 
\end{split}
\end{align}
\end{proposition}

We will prove Proposition~\ref{prop::onelog_CR_fourier} in Section~\ref{subsec::onelog_CR_fourier}. When $\gamma=0$, for $\beta\in\R$ with $\beta<1/2$, the moment-generating function~\eqref{eqn::onelog_CR_fourier} becomes
\begin{align}\label{eqn::CR_fourier}
\begin{split}
\E\left[\left(\frac{\CR(\Omega; u)}{\CR(\Omega\setminus\bs{\eta}; u)}\right)^{\beta}\one\{u\in\Omega_{\bs{\eta}}^+\}\right]
=&\frac{\sin\left(\Gamma_0( u)\sqrt{2\beta}\right)}{\sin\left(\pi\sqrt{2\beta}\right)}. 
\end{split}
\end{align}
The expectation in the LHS of~\eqref{eqn::CR_fourier} were considered for multiple $\SLE_{\kappa}$ in earlier literature.  
When $N=1$, denote by $\PP^{(\kappa)}_1$ the law of chordal $\SLE_{\kappa}$ in $(\Omega; x_1, x_2)$. It was derived in~\cite{SchrammSheffieldWilsonConformalRadii} that the expectation 
$\E^{(\kappa)}_1\left[\CR(\Omega\setminus\bs{\eta}; u)^{-\beta}\right]$
is a hypergeometric function when $\beta<1-\kappa/8$. 
Suppose $\PP^{(\kappa)}_N$ is a probability measure on $N$-chordal $\SLE_{\kappa}$ in the polygon $(\Omega; \bs{x})$ (see details in~\cite[Section~2]{FengWuRadialBPZFKIsing}).
The authors in~\cite[Section~2]{FengWuRadialBPZFKIsing} proved that the  expectation $\E^{(\kappa)}_N\left[\CR(\Omega\setminus\bs{\eta}; u)^{-\beta}\right]$ is finite for $\kappa\le 6$ as long as $\beta<1-\kappa/8$. 
However, it is usually hard to derive an explicit formula for such expectation when $N\ge 2$. Proposition~\ref{prop::onelog_CR_fourier} and~\eqref{eqn::CR_fourier} give an explicit expression for such expectation when $\kappa=4$ using the observation in Theorem~\ref{thm::multiplelog_polygon}. This is very special for $\kappa=4$. Moreover, the RHS of~\eqref{eqn::CR_fourier} also tells that the expectation in the LHS of~\eqref{eqn::CR_fourier} explodes at rate $(1-2\beta)^{-1}$ when $\beta\to 1/2$. 

\begin{proposition}\label{prop::twopoint_correlation}
Assume the same notation as in Theorem~\ref{thm::multiplelog_polygon}  with $M=2$.
Fix distinct $u_1, u_2\in\Omega$ and $\gamma_1=\gamma_2=\gamma\in\R$. For $\xi\in\C$ with $\Re{\xi}\le\gamma^2/2$, we have
\begin{align}\label{eqn::twopoint_correlation}
\begin{split}
&\E_{\bs{\gamma}}\left[\left(\frac{\CR(\Omega; u_1)}{\CR(\Omega\setminus\bs{\eta}; u_1)}\right)^{\xi}\left(\frac{\CR(\Omega; u_2)}{\CR(\Omega\setminus\bs{\eta}; u_2)}\right)^{\xi}\left(\one\{u_1\in\Omega_{\bs{\eta}}^+\}-\one\{u_2\in\Omega_{\bs{\eta}}^+\}\right)\right]\\
=&\exp\left(2(\xi-\gamma^2)\Green(\Omega; u_1, u_2)+\gamma(\pi-\Gamma_0(u_1)-\Gamma_0(u_2))\right)\frac{\sin\left((\Gamma_0( u_1)-\Gamma_0( u_2))\sqrt{2\xi-\gamma^2}\right)}{\sin\left(\pi\sqrt{2\xi-\gamma^2}\right)}. 
\end{split}
\end{align}
\end{proposition}

We will prove Proposition~\ref{prop::twopoint_correlation} in Section~\ref{subsec::twopoint_correlation}. The formulae in Propositions~\ref{prop::onelog_CR_fourier}-\ref{prop::twopoint_correlation} are closely related to correlation functions in CFT. Conformal radius is essentially the one-point function of a primary field. 
Thus, Eq.~\eqref{eqn::CR_fourier} can be viewed as the one-point function of a vertex operator in a random SLE geometry, at the free-boson point ($\kappa=4$). 
The LHS of~\eqref{eqn::twopoint_correlation} is not the standard CFT two-point function. Instead, it is a two-point function with a difference of indicator functions. Eq.~\eqref{eqn::twopoint_correlation} can be viewed as two-point function of vertex operators in a random SLE geometry with a specific projection that selects configuration where the SLE curves separate the two points.

\subsection{Consequences: asymptotic}

\begin{proposition}\label{prop::asymp}
Assume the same notation as in Theorem~\ref{thm::multiplelog_polygon}  with $M=1$.
Fix $u_1=u\in\Omega$ and $\gamma_1=\gamma\in\R$. We have\footnote{
For two functions $F(\eps)$ and $G(\eps)$, 
we write $F(\eps)\sim G(\eps)$ if $\lim_{\eps\to 0}F(\eps)/G(\eps)\in (0,\infty)$.}
\begin{align}
&\PP_{\gamma}\left[\forall j\in \{1, \ldots, N\}, \dist(\eta^j, u)\le \eps\right]\sim \eps^{\frac{N^2+\gamma^2}{2}}\frac{\prod_{j=1}^{2N}\Poisson(\Omega; u, x_j)^{\frac{N}{2}}\exp(-\gamma\Gamma_0(\Omega; \bs{x}; u))}{\CR(\Omega; u)^{\frac{N^2+\gamma^2}{2}}\prod_{\substack{1\le i<j\le 2N\\ i\not\equiv j\mod 2}}\Poisson(\Omega; x_i, x_j)^{\frac{1}{2}}};
\label{eqn::asymp_forall}
\end{align}
where $\CR(\Omega; u)$ is conformal radius (see~\eqref{eqn::CR_H}-\eqref{eqn::CR_cov}), $\Poisson(\Omega; u, x_j)$ is Poisson kernel (see~\eqref{eqn::Poisson_H}-\eqref{eqn::Poisson_cov}), $\Poisson(\Omega; x_i, x_j)$ is boundary Poisson kernel (see~\eqref{eqn::bPoisson_H}-\eqref{eqn::bPoisson_cov}), and
the constant in $\sim$ depends on $(N, \gamma)$. 
\end{proposition}

We will prove Proposition~\ref{prop::asymp} for general $\gamma\in\R$ in Section~\ref{sec::asymp}. 
The proof relies on tools ``common-time parameter" and ``multi-time martingale" developed recently from~\cite{HealeyLawlerNSidedRadialSLE, HuangPeltolaWuMultiradialSLE}.
If we denote by $\PP=\PP_0$ the measure when $\gamma=0$, we find 
\begin{align*}
\frac{\PP_{\gamma}\left[\forall j\in \{1, \ldots, N\}, \dist(\eta^j, u)\le \eps\right]}{\PP\left[\forall j\in \{1, \ldots, N\}, \dist(\eta^j, u)\le \eps\right]}\sim \eps^{\frac{\gamma^2}{2}}\frac{\exp(-\gamma\Gamma_0(\Omega; \bs{x}; u))}{\CR(\Omega; u)^{\frac{\gamma^2}{2}}}. 
\end{align*}
The extra $\frac{\gamma^2}{2}$ in the exponent comes from the influence of $\gamma\Green(\Omega; u, \cdot)$. As Green's function explodes as $-\log |u-z|$ when $z$ is close to $u$, this singularity makes it harder for the level lines to get close to $u$ and the probability is smaller by $\eps^{\frac{\gamma^2}{2}}$. 

\paragraph*{Outline.}
We give preliminaries in Section~\ref{sec::pre}. 
We prove Theorem~\ref{thm::multiplelog_polygon} and Propositions~\ref{prop::onelog_CR_fourier}-\ref{prop::twopoint_correlation} in Section~\ref{sec::GFF_levellines}. 
We prove Theorem~\ref{thm::Brownian_represent} in Section~\ref{sec::Brownian_represent}. 
We prove Proposition~\ref{prop::asymp} in Section~\ref{sec::asymp}. 

\paragraph*{Acknowledgment.}
The authors thank Fredrick Viklund and Yilin Wang for helpful discussion on the Hele-Shaw flow. 
H.W. is supported by New Cornerstone Investigator Program 100001127. 
H.W. is partly affiliated at Yanqi Lake Beijing Institute of Mathematical Sciences and Applications, Beijing, China.

\section{Preliminaries}
\label{sec::pre}
\subsection{Conformal radius, Green's function and Poisson kernel}
\paragraph*{Conformal radius.}
For a simply connected domain $\Omega\subsetneq\C$ and $z\in \Omega$, the \emph{conformal radius} of $\Omega$ seen from $z$ is defined by $\CR(\Omega; z) := 1/\phi'(z)$  where $\phi : \Omega \to \U$ is the conformal map such that $\phi(z)=0$ and $\phi'(z)>0$. 
In particular, for $\Omega=\HH$ and $z\in\HH$, we have
\begin{align}\label{eqn::CR_H}
\CR(\HH; z)=2\Im(z). 
\end{align}
Conformal radius is conformally covariant: for any conformal map $\varphi$ on $\Omega$, we have
\begin{align} \label{eqn::CR_cov}
\CR(\Omega; z)=|\varphi'(z)|^{-1}\CR(\varphi(\Omega); \varphi(z)).
\end{align}
Conformal radius can be bounded by Euclidean distance from both sides 
\begin{align}\label{eqn::CR_Euc}
\dist(z, \partial\Omega)\le \CR(\Omega; z)\le 4\dist(z, \partial\Omega)
\end{align}
due to Schwarz lemma and Koebe's one quarter theorem. 

\paragraph*{Green's function.}
For $\HH$, the \textit{Green's function} is defined as\footnote{Green's function can be defined by different multiplicative constant in different literature. The normalization we used here has the advantage that the harmonic function in~\eqref{eqn::multiplelog_harmonic_H} does not need a multiplicative constant.}
\begin{equation}\label{eqn::Green_H}
    \Green(\HH;z,w)=\log{\frac{|z-\overline{w}|}{|z-{w}|}}, \qquad z,w\in \HH.
\end{equation}
For a simply connected domain $\Omega\subsetneq \C$, we define the Green's function via conformal invariance:
\begin{equation}\label{eqn::Green_inv}
    \Green(\Omega;z,w)=\Green(\HH;\varphi(z),\varphi(w)), \qquad z,w\in \Omega,
\end{equation}
where $\varphi$ is any conformal map from $\Omega$ onto $\HH$.

Green's function can be viewed as $-2\pi\Delta^{-1}$ with Dirichlet boundary condition: for any $f\in C_0^{\infty}(\Omega)$, the following function 
\[-\frac{1}{2\pi}\int_{\Omega} \Green(\Omega; \cdot, w)f(w)\ud w\]
is a smooth function in $\Omega$ whose Laplacian is $f$ and whose boundary value on $\partial \Omega$ is zero.

\paragraph*{Poisson kernel.} 
For a $1$-polygon $(\HH; x)$ with $z\in\HH$, the \textit{Poisson kernel} is defined as 
\begin{equation}\label{eqn::Poisson_H}
\Poisson(\HH;z,x) :=\lim_{\eps\to 0}\frac{1}{2\pi\eps}\Green(\HH; z, x+\ii\eps)= \frac{\Im z}{\pi|z - x|^2}.
\end{equation}
For a general $1$-polygon $(\Omega; x)$ with $z\in\Omega$, we define the Poisson kernel via conformal covariance: 
\begin{equation}\label{eqn::Poisson_cov}
\Poisson(\Omega;z,x) = |\varphi'(x)|  \Poisson(\HH;\varphi(z),\varphi(x)),
\end{equation}
where $\varphi$ is any conformal map from $\Omega$ onto $\HH$.

\paragraph*{Boundary Poisson kernel.}
For a $2$-polygon $(\HH; x_1, x_2)$ with $x_1<x_2$, the \textit{boundary Poisson kernel} is defined as
\begin{align}\label{eqn::bPoisson_H}
\Poisson(\HH; x,y) :=\lim_{\eps\to 0}\frac{1}{\eps}\Poisson(\HH; x+\ii\eps, y)=\frac{1}{\pi (x_2-x_1)^2}. 
\end{align}
For a general $2$-polygon $(\Omega; x_1, x_2)$, we define the boundary Poisson kernel via conformal covariance: 
\begin{align}\label{eqn::bPoisson_cov}
\Poisson(\Omega; x_1, x_2) := |\varphi'(x_1)\varphi'(x_2)|\Poisson(\HH; \varphi(x_1),\varphi(x_2)), 
\end{align}
where $\varphi$ is any conformal map from $\Omega$ onto $\HH$ with $\varphi(x_1)<\varphi(x_2)$. 
In particular, for the unit disc $\U$ with $\theta_1<\theta_2<\theta_1+2\pi$, 

\subsection{Measures on Brownian paths}
\label{subsec::Brownianpath}
In this section, we introduce measures on Brownian paths mentioned in Theorem~\ref{thm::Brownian_represent}.  
We follow the same normalization as in~\cite{LawlerWernerBrownianLoopsoup}. 
Let $\bm(\C; z, \cdot; t)$ denote the law of a standard two-dimensional Brownian motion starting from $z$. We write 
\[\bm(\C; z, \cdot; t)=\int_{\C}\ud A(w)\;\bm(\C; z, w; t),\]
where $A$ denotes area and $\bm(\C; z, w; t)$ is a measure supported on paths from $z$ to $w$ in time $t$. The measure $\bm(\C; z, w)$ is defined by 
\[\bm(\C; z, w)=\int_0^{\infty}\ud t\;\bm(\C; z, w; t). \]

\paragraph*{Brownian bridge and Green's function.}
If $\Omega$ is a domain and $z,w\in\Omega$, we define $\bm(\Omega; z, w)$ to be $\bm(\C; z, w)$ restricted to paths contained in $\Omega$. 
If $\Omega$ is such that a Brownian motion in $\Omega$ eventually exits $\Omega$ and 
$z\neq w$, then the total mass of $\bm(\Omega; z, w)$ is finite which is the same as $\frac{1}{\pi}$ times Green's function~\eqref{eqn::Green_H}-\eqref{eqn::Green_inv}. 
We write 
\begin{align}\label{eqn::BM_zw}
\bm(\Omega; z, w)[\cdot]=\frac{1}{\pi}\Green(\Omega; z, w)\bm^{\#}(\Omega; z, w)[\cdot], 
\end{align}
where $\bm^{\#}(\Omega; z, w)$ is a probability measure on paths in $\Omega$ from $z$ to $w$. 
We call $\bm(\Omega; z, w)$ the Brownian bridge measure.
If $w=z$, $\bm(\Omega; z, z)$ is a $\sigma$-finite measure whose total mass is infinite. 
We call $\bm(\Omega; z, z)$ the Brownian bubble measure.

\paragraph*{Brownian bridge and Poisson kernel.}
Suppose $\Omega$ is simply connected and $\partial\Omega$ is piecewise analytic. Let $B$ be a Brownian motion starting at $z\in\Omega$ and stopped at its exit time $\tau_{\Omega}=\inf\{t: B_t\not\in\Omega\}$. Define $\bm(\Omega; z, \partial\Omega)$ to be the law of $(B_t)_{0\le t\le \tau_{\Omega}}$. As $\partial\Omega$ is piecewise analytic, we can write 
\[\bm(\Omega; z, \partial\Omega)=\int_{\partial\Omega}\ud y \;\bm(\Omega; z, y),\]
where $\bm(\Omega; z, y)$ for $z\in \Omega$ and $y\in\partial\Omega$ denotes a measure on paths in $\Omega$ from $z$ to $y$. The total mass of $\bm(\Omega; z, y)$ is Poisson kernel~\eqref{eqn::Poisson_H}-\eqref{eqn::Poisson_cov}. We write 
\begin{align}\label{eqn::BM_zy}
\bm(\Omega; z, y)[\cdot]=\Poisson(\Omega; z, y)\bm^{\#}(\Omega; z, y)[\cdot], 
\end{align}
where $\bm^{\#}(\Omega; z, y)$ is a probability measure on paths in $\Omega$ starting from $z$ and hitting $\partial\Omega$ at $y$. 
We also call $\bm(\Omega; z, y)$ the Brownian bridge measure.

\paragraph*{Brownian excursion and boundary Poisson kernel.}
For $x, y\in\partial\Omega$, Brownian excursion measure $\bm(\Omega; x, y)$ can be defined as 
\[\bm(\Omega; x, y)=\lim_{\eps\to 0}\frac{1}{\eps}\bm(\Omega; x+\eps\bs{n}_x, y).\]
The total mass of $\bm(\Omega; x, y)$ is boundary Poisson kernel~\eqref{eqn::bPoisson_H}-\eqref{eqn::bPoisson_cov}. We write 
\begin{align}\label{eqn::BM_xy}
\bm(\Omega; x, y)[\cdot]=\Poisson(\Omega; x, y)\bm^{\#}(\Omega;x, y)[\cdot],
\end{align}
where $\bm^{\#}(\Omega;x, y)$ is a probability measure on paths in $\Omega$ from $x$ to $y$. 

\paragraph*{Brownian loop measure.}
Brownian loop measure $\blm$ is a measure on unrooted loops which is defined by
\[\blm=\int_{\C}\ud A(z)\int_0^{\infty}\frac{1}{2\pi t^2}\bm^{\#}(\C; z,z; t).\]
For a domain $\Omega$, we define Brownian loop measure in $\Omega$ to be $\blm$ restricted to the curves in $\Omega$. 
By construction, Brownian loop measure satisfies the restriction property. It also satisfies conformal invariance, see~\cite[Proposition~6]{LawlerWernerBrownianLoopsoup}.
\section{Level lines of Gaussian free field}
\label{sec::GFF_levellines}
We will prove Theorem~\ref{thm::multiplelog_polygon} and Propositions~\ref{prop::onelog_CR_fourier} and~\ref{prop::twopoint_correlation} in this section. 
We first give preliminaries of GFF and Loewner chains in Section~\ref{subsec::pre_GFF}. 
Then, we derive the coupling between GFF and level lines in Section~\ref{subsec::GFF_levellines} where the martingale calculation in Lemma~\ref{lem::multiple_log_mart} plays the key role. We complete the proof of Theorem~\ref{thm::multiplelog_polygon} in Section~\ref{subsec::GFF_cascade} by a careful cascade analysis.
Using the Radon-Nikodym derivative in Theorem~\ref{thm::multiplelog_polygon}, we complete the proof of Proposition~\ref{prop::onelog_CR_fourier} in Section~\ref{subsec::onelog_CR_fourier} and complete the proof of Proposition~\ref{prop::twopoint_correlation} in Section~\ref{subsec::twopoint_correlation}. 
 
\subsection{GFF and partition functions}
\label{subsec::pre_GFF}

\paragraph*{Gaussian free field (GFF).}
We will give a very brief summary on zero-boundary Gaussian free field (GFF), more detailed information can be found in~\cite{SheffieldGFFMath, SchrammSheffieldContinuumGFF, MillerSheffieldIG1, WangWuLevellinesGFFI}. 
Suppose $\Omega\subset\C$ is a domain with harmonically non-trivial boundary. For $f,g\in L^2(\Omega)$, we denote by $(f,g)$ the inner product in $L^2(\Omega)$: 
\[(f,g)=\int_\Omega f(z)g(z)\ud z.\]
The Sobolev space $H_0^1(\Omega)$ is the Hilbert space closure of $C_0^{\infty}(\Omega)$ with respect to the Dirichlet inner product
\[(f,g)_{\nabla}=\frac{1}{2\pi}\int \nabla f\cdot \nabla g \; \ud z. \]

The zero-boundary GFF $\Gamma$ on $\Omega$ is given by 
$\Gamma=\sum_n X_n f_n$
where $\{X_n\}$ is a sequence of i.i.d. normal Gaussian random variables and $\{f_n\}$ is an orthonormal basis for $H_0^1(\Omega)$. Such a sum does not converge in $H_0^1(\Omega)$, but it does converge in an appropriate space of distributions: the limit $\sum_n X_n(f_n, g)$ almost surely exists for all $g\in C_0^{\infty}(\Omega)$, and the limit value, denote by $(\Gamma, g)$, as a function of $g$ is almost surely a continuous functional on $C_0^{\infty}(\Omega)$. For any $g\in C_0^{\infty}(\Omega)$, $(\Gamma, g)$ is a mean-zero Gaussian with variance 
\[\iint_{\Omega\times \Omega}g(z)g(w)\Green(\Omega; z,w)\ud z\ud w.\] 

The GFF with boundary data $\Gamma_0$ is the sum of the zero-boundary GFF on $\Omega$ and the function in $\Omega$, still denoted by $\Gamma_0$, which is harmonic in $\Omega$ and is equal to $\Gamma_0$ on $\partial \Omega$. 

\paragraph*{Partition functions.}
The law of the level line in Theorem~\ref{thm::multiplelog_polygon} is encoded by its partition function which we define in this paragraph. We assume the same notations as in Theorem~\ref{thm::multiplelog_polygon} and assume $\Omega=\HH$. 
Fix $N\ge 1$ and $\bs{x}=(x_1, \ldots, x_{2N})$ with $x_1<\cdots<x_{2N}$. The harmonic function $\Gamma_0$ in Theorem~\ref{thm::multiplelog_polygon} has the following expression when $\Omega=\HH$: 
\begin{equation}\label{eqn::multiplelog_harmonic_H}
\Gamma_0(\HH; \bs{x}; z)=\sum_{1\le k\le 2N}(-1)^k\arg(z-x_k),\quad\text{ for }z\in\HH.
\end{equation}
For $\bs{u}=(u_1, \ldots, u_M)\in\HH^M$ where $u_1, \ldots, u_M$ are distinct, we define partition function for $\Omega=\HH$ by 
\begin{align}\label{eqn::multiplelog_pf_H}
\LZgff^{(\bs{\gamma})}(\HH; \bs{x};\bs{u}):=\underbrace{\prod_{1\le i<k\le 2N}(x_k-x_i)^{\frac{1}{2}(-1)^{k-i}}}_{\LZgff(\HH; \bs{x}):=}& \times\exp\left(\LE_{\bs{\gamma}}(\HH; \bs{u})+\sum_{\ell=1}^M\gamma_\ell\Gamma_0(\HH; \bs{x}; u_\ell)\right),
\end{align}
where $\LE_{\bs{\gamma}}$ is reduced Green's energy~\eqref{eqn::reducedGreenenergy_def}. 
Suppose $(\Omega; \bs{x})$ is $2N$-polygon with $\bs{u}=(u_1, \ldots u_M)$ where $u_1, \ldots, u_M\in \Omega$ are distinct. 
We extend the partition function $\LZ_{\GFF}^{(\bs{\gamma})}$ in~\eqref{eqn::multiplelog_pf_H} to $(\Omega; \bs{x})$ with $\bs{u}$ via conformal covariance:
\begin{align}\label{eqn::multiplelog_pf_cov}
\LZ_{\GFF}(\Omega; \bs{x}; \bs{u})=&\prod_{j=1}^{2N}|\varphi'(x_j)|^{1/4}\times\LZ_{\GFF}(\HH; \varphi(\bs{x}); \varphi(\bs{u}));\\
\LZ_{\GFF}^{(\bs{\gamma})}(\Omega; \bs{x}; \bs{u})=&\prod_{j=1}^{2N}|\varphi'(x_j)|^{1/4}\times\prod_{\ell=1}^M|\varphi'(u_{\ell})|^{-\frac{\gamma_{\ell}^2}{2}}\times\LZ_{\GFF}^{(\bs{\gamma})}(\HH; \varphi(\bs{x}); \varphi(\bs{u})); 
\end{align}
where $\varphi$ and any conformal map from $\Omega$ onto $\HH$ with $\varphi(x_1)<\cdots<\varphi(x_{2N})$. 
We may also write
\begin{align*}
\LZ_{\GFF}^{(\bs{\gamma})}(\Omega; \bs{x}; \bs{u})=\LZ_{\GFF}(\Omega; \bs{x})\times \exp\left(\LE_{\bs{\gamma}}(\Omega; \bs{u})+\sum_{\ell=1}^M\gamma_\ell\Gamma_0(\Omega; \bs{x}; u_\ell)\right).
\end{align*}
In particular, when $N=0$, we have 
\begin{align}
\LZ_{\GFF}(\Omega; \emptyset)=1, \qquad \LZ_{\GFF}^{(\bs{\gamma})}(\Omega; \emptyset; \bs{u})=\exp\left(\LE_{\bs{\gamma}}(\Omega; \bs{u})\right). 
\end{align}


\paragraph*{Loewner chain.}
Fix $x\in\R$ and $T>0$. Let $\eta: [0,T]\to\overline{\HH}$ be a continuous simple curve such that $\eta_0=x$ and $\eta_{(0,T)}\subset\HH$. For each $t\in[0,T]$, let $g_t$ be the 
unique conformal map from $\HH\setminus\eta_{[0,t]}$ onto $\HH$ with the normalization $\lim_{z\to\infty}|g_t(z)-z|=0$. 
We say that $\eta$ is parameterized by the half-plane capacity if, for all $t\in [0,T]$, 
\[g_t(z)=z+\frac{2t}{z}+o(1/|z|), \qquad\text{as }z\to\infty.\]
These maps $\{g_t: t\in [0,T]\}$ solve the Loewner equation
\begin{equation*}
	\partial_t g_t(z)=\frac{2}{g_t(z)-W_t}, \qquad g_0(z)=z,
\end{equation*}
where $W: [0,T]\to \R$ is a continuous function called the driving function of $\eta$. 

Fix $\kappa>0$ and $x\in\R$. The process $\SLE_{\kappa}$ in $(\HH; x, \infty)$ is defined as Loewner chain with driving function $W_t=\sqrt{\kappa}B_t+x$ where $B_t$ is standard one-dimensional Brownian motion. When $\kappa\in (0,4]$, it is almost surely a continuous simple curve. For a $2$-polygon $(\Omega; x_1, x_2)$, let $\varphi: \Omega\to\HH$ be a conformal map with $\varphi(x_2)=\infty$, then $\SLE_{\kappa}$ in $(\Omega; x_1, x_2)$ is defined as the image of $\SLE_{\kappa}$ in $(\HH; \varphi(x_1), \infty)$ under $\varphi^{-1}$. 

\subsection{Level lines of GFF}
\label{subsec::GFF_levellines}
\begin{lemma}\label{lem::multiple_log_mart}
Assume the same notation as in Theorem~\ref{thm::multiplelog_polygon} and assume $\Omega=\HH$ and $x_1<\cdots<x_{2N}$. 
Recall that $\LZgff^{(\bs{\gamma})}(\HH; \bs{x};\bs{u})$ is defined in~\eqref{eqn::multiplelog_pf_H}. 
The following process is a local martingale for $\SLE_4$ in $(\HH; x_{2j-1}, \infty)$: 
\begin{align}\label{eqn::multiple_log_mart}
\begin{split}
M_t(\LZgff^{(\bs{\gamma})}):=&\prod_{\substack{1\le i\le 2N\\i\neq 2j-1}}g_t'(x_i)^{\frac{1}{4}}\times \prod_{\ell=1}^M |g_t'(u_{\ell})|^{-\frac{\gamma_{\ell}^2}{2}}\times\LZgff^{(\bs{\gamma})}(\HH; \bs{X}_t; g_t(\bs{u})),
\end{split}
\end{align} 
where
$\bs{X}_t=(g_t(x_1), \ldots, g_t(x_{2j-2}), W_t, g_t(x_{2j}), \ldots, g_t(x_{2N}))$ and $g_t(\bs{u})=(g_t(u_1), \ldots, g_t(u_M))$.
Moreover, suppose $\eta$ has the same law as $\SLE_4$ in $(\HH; x_{2j-1}, \infty)$ weighted by the local martingale~\eqref{eqn::multiple_log_mart}, then its driving function $W_t$ solves
\begin{equation}\label{eqn::multiple_log_SDE}
\ud W_t=2\ud B_t+\sum_{i\neq 2j-1}\frac{2(-1)^{i}}{g_t(x_i)-W_t}\ud t+\sum_{1\le\ell\le M}\Im\frac{4\gamma_\ell}{g_t(u_\ell)-W_t}\ud t. 
\end{equation}
\end{lemma}

\begin{proof}
To simplify the notation, we assume $j=1$ and the other cases can be proved similarly. 
Denote by $\PP(\LZ_{\GFF})$ the law of $\SLE_4(-2, +2, \ldots, -2)$ from $x_{1}$ to $\infty$ with force points $(x_2, \ldots, x_{2N})$, and the force weight of $x_i$ equals $2(-1)^{i-1}$ for $2\le i\le 2N$. 
By~\cite[Theorem 6]{SchrammWilsonSLECoordinatechanges}, the process $\{M_t(\LZgff)\}_{t\ge 0}$ is a local martingale for $\SLE_4$ in $(\HH; x_{1}, \infty)$. 
Moreover, the law of $\SLE_4$ in $(\HH; x_{1}, \infty)$ weighted by the local martingale $M_t(\LZgff)$ is the same as the law of $\PP(\LZ_{\GFF})$ whose driving function solves
\begin{equation}\label{eqn::GFF_alternating_SDE}
\ud W_t=2\ud B_t+\sum_{i=2}^{2N}\frac{2(-1)^{i}}{g_t(x_i)-W_t}\ud t. 
\end{equation}
Thus, it suffices to prove that $\{M_t(\LZgff^{(\bs{\gamma})})/M_t(\LZgff)\}_{t\ge 0}$ is a local martingale under $\PP(\LZ_{\GFF})$. 

\medbreak
Denote $\Gamma^0_t(z):=\Gamma_0(\HH; W_t,g_t(x_2),\ldots,g_t(x_{2N}); g_t(z))$. From the definition in~\eqref{eqn::multiplelog_harmonic_H}, we have
\begin{equation}\label{eqn::coupling_calculation_aux1}
\ud \Gamma_t^0(z)=\Im\frac{1}{g_t(z)-W_t}\left(\ud W_t-\sum_{i=2}^{2N}\frac{2(-1)^{i}}{g_t(x_i)-W_t}\ud t\right).
\end{equation}
For the Green's function in $\HH$~\eqref{eqn::Green_H}, we have
\begin{equation}\label{eqn::coupling_calculation_aux2}
\ud \Green(\HH; g_t(z), g_t(w))=-4\Im\frac{1}{g_t(z)-W_t}\Im\frac{1}{g_t(w)-W_t}\ud t.
\end{equation}
For the conformal radius in $\HH$~\eqref{eqn::CR_H}, we have 
\begin{equation}\label{eqn::coupling_calculation_aux3}
\ud \left(\log\CR(\HH; g_t(z))-\log|g_t'(z)|\right)=-4\left(\Im\frac{1}{g_t(z)-W_t}\right)^2\ud t. 
\end{equation}
Denote 
\begin{align}\label{eqn::multiple_log_SDE_aux1}
\begin{split}
L_t:=\log\frac{M_t(\LZgff^{(\bs{\gamma})})}{M_t(\LZgff)}=&\sum_{1\le i<k\le M} \gamma_i \gamma_k \Green(\HH;g_t(u_i),g_t(u_k))\\
&+\sum_{1\le \ell\le M}\frac{\gamma_{\ell}^2}{2}\left(\log\CR(\HH; g_t(u_{\ell}))-\log|g_t'(u_{\ell})|\right)+\sum_{1\le \ell\le M}\gamma_{\ell}\Gamma_t^0(u_{\ell}). 
\end{split}
\end{align}
Plugging~\eqref{eqn::GFF_alternating_SDE},\eqref{eqn::coupling_calculation_aux1},\eqref{eqn::coupling_calculation_aux2} and~\eqref{eqn::coupling_calculation_aux3} into~\eqref{eqn::multiple_log_SDE_aux1}, we have
\begin{align*}
\ud L_t=&-4\sum_{1\le i<k\le M}\Im\frac{\gamma_i}{g_t(u_i)-W_t}\Im\frac{\gamma_k}{g_t(u_k)-W_t}\ud t\\
&-2\sum_{1\le \ell\le M}\left(\Im\frac{\gamma_{\ell}}{g_t(u_{\ell})-W_t}\right)^2\ud t+\sum_{1\le \ell\le M}\Im\frac{2\gamma_{\ell}}{g_t(u_{\ell})-W_t}\ud B_t, \\
\ud\langle L, L\rangle_t=&\left(\sum_{1\le \ell\le M}\Im\frac{2\gamma_{\ell}}{g_t(u_{\ell})-W_t}\right)^2\ud t. 
\end{align*}
Therefore,
\begin{align}\label{eqn::multiple_log_SDE_aux2}
\ud \left(\frac{M_t(\LZgff^{(\bs{\gamma})})}{M_t(\LZgff)}\right)=\frac{M_t(\LZgff^{(\bs{\gamma})})}{M_t(\LZgff)}\left(\ud L_t+\frac{1}{2}\ud\langle L, L\rangle_t\right)=\frac{M_t(\LZgff^{(\bs{\gamma})})}{M_t(\LZgff)}\left(\sum_{1\le\ell\le M}\Im\frac{2\gamma_\ell}{g_t(u_\ell)- W_t}\right)\ud B_t.
\end{align}
This implies that $\{M_t(\LZgff^{(\bs{\gamma})})/M_t(\LZgff)\}_{t\ge 0}$ is a local martingale under $\PP(\LZ_{\GFF})$. 
Combining~\eqref{eqn::GFF_alternating_SDE} and~\eqref{eqn::multiple_log_SDE_aux2}, we obtain~\eqref{eqn::multiple_log_SDE} due to Girsanov's theorem.
\end{proof}

\begin{lemma}\label{lem::GFFlog_levelline_coupling}
Assume the same notation as in Theorem~\ref{thm::multiplelog_polygon} and assume $\Omega=\HH$ and $x_1<\cdots<x_{2N}$. 
For $j\in\{1, \ldots, N\}$, the level line of 
$\Gamma+\sum_{\ell=1}^M \gamma_{\ell} \Green(\HH; u_{\ell}, \cdot)$
starting from $x_{2j-1}$ has the same law as $\SLE_4$ in $(\HH; x_{2j-1}, \infty)$ weighted by the local martingale~\eqref{eqn::multiple_log_mart}, up to $\sigma$ which is the first time any one of $\{x_{2j-2}, x_{2j}, u_1, \ldots, u_M\}$ is disconnected from $\infty$. 
\end{lemma}

\begin{proof}
To simplify the notation, we assume $j=1$ and the other cases can be proved similarly.
Denote by $\PP(\LZ_{\GFF}^{(\bs{\gamma})})$ the law of $\SLE_4$ in $(\HH; x_{2j-1}, \infty)$ weighted by the local martingale~\eqref{eqn::multiple_log_mart} as in Lemma~\ref{lem::multiple_log_mart}. 
Define $\Gamma^0_t(z):=\Gamma_0(\HH; W_t,g_t(x_2),\ldots,g_t(x_{2N}); g_t(z))$ for every $z\in\HH\setminus \eta[0,t]$ and $0\le t<\sigma$ and define 
\begin{equation*}
N_t(z):=\Gamma_t^0(z)+\sum_{1\le\ell\le M}\gamma_\ell \Green(\HH;g_t(z),g_t(u_\ell)).
\end{equation*} 
For every smooth function $f$ with compact support in $\HH$, define
\begin{equation}\label{eqn::mart_multi_test}
M_t:=\exp\left(\int_\HH N_t(z)f(z)\ud z+\frac{1}{2}\iint_{\HH\times\HH} \Green(\HH; g_t(z),g_t(w))f(z)f(w) \ud z\ud w\right).
\end{equation}
By the same arguments in~\cite[Proof of Theorem~1.1]{MillerSheffieldIG1} (see also~\cite[Proposition~3.1]{MillerSheffieldIG4} and~\cite[Proposition~2.2.7]{WangWuLevellinesGFFI}), to obtain the coupling $(\Gamma+\sum_{\ell=1}^M \gamma_{\ell} \Green(\HH; u_{\ell}, \cdot), \eta)$ in the statement, it suffices to prove that $M_t$ in~\eqref{eqn::mart_multi_test} is a local martingale under $\PP(\LZ_{\GFF}^{(\bs{\gamma})})$. 
This is true due to the following two facts. 
\begin{itemize}
\item For every $z\in\HH$, from~\eqref{eqn::multiple_log_SDE}, \eqref{eqn::coupling_calculation_aux1} and~\eqref{eqn::coupling_calculation_aux2}, we obtain
\begin{align}\label{eqn::coupling_calculation_aux5}
\ud N_t(z)=\Im\frac{2}{g_t(z)-W_t}\ud B_t.
\end{align}
Thus, $\{N_t(z)\}_{t\ge 0}$ is a local martingale under $\PP(\LZ_{\GFF}^{(\bs{\gamma})})$. 
\item For distinct points $z,w\in\HH$, from~\eqref{eqn::coupling_calculation_aux2} and~\eqref{eqn::coupling_calculation_aux5}, 
\begin{equation*}
\ud\langle N_t(z), N_t(w)\rangle+\ud \Green(\HH; g_t(z),g_t(w))=0.
\end{equation*}
\end{itemize} 
\end{proof}

\subsection{Cascade relation and proof of Theorem~\ref{thm::multiplelog_polygon}}
\label{subsec::GFF_cascade}
\begin{lemma}\label{lem::cascade_singlecurve}
Assume the same notation as in Theorem~\ref{thm::multiplelog_polygon} and assume $\Omega=\HH$ and $x_1<\cdots<x_{2N}$. 
We denote $\bs{x}=(x_1, \ldots, x_{2N})$ and $\bs{u}=(u_1, \ldots, u_M)$. 
Let $\eta$ be the level line of $\Gamma+\sum_{\ell=1}^M\gamma_{\ell}\Green(\HH; u_{\ell}, \cdot)$ starting from $x_1$ and denote by $\sigma$ the first time $\eta$ disconnects $x_2$ from $\infty$. Denote by $H_{\eta}^R$ (resp. $H_{\eta}^L$) the connected component of $\HH\setminus\eta$ that is to the right of $\eta$ (resp. to the left of $\eta$). 
Denote $\bs{u}^R=\{u_1, \ldots, u_M\}\cap H_{\eta}^R$ and $\bs{u}^L=\{u_1, \ldots, u_M\}\cap H_{\eta}^L$ and let $\bs{\gamma}^R$ be the collection of $\gamma_i$ such that $u_i\in \bs{u}^R$ and $\bs{\gamma}^L$ be the collection of $\gamma_i$ such that $u_i\in \bs{u}^L$.
Then the law of $\eta$ under $\PP_{\bs{\gamma}}$ is absolutely continuous with respect to its law under $\PP=\PP_{\bs{0}}$ and the Radon-Nikodym derivative is given by 
\begin{align}\label{eqn::cascade_singlecurve_RN}
\frac{\LZ_{\GFF}(\HH; \bs{x})}{\LZ_{\GFF}^{(\bs{\gamma})}(\HH; \bs{x}; \bs{u})}\sum_{j=1}^N\one\{\eta(\sigma)=x_{2j}\}\frac{\LZ_{\GFF}^{(\bs{\gamma}^R)}(H_{\eta}^R; x_3, \ldots, x_{2j-1}, x_2; \bs{u}^R)}{\LZ_{\GFF}(H_{\eta}^R; x_3, \ldots, x_{2j-1}, x_2)}\frac{\LZ_{\GFF}^{(\bs{\gamma}^L)}(H_{\eta}^L; x_{2j+1}, \ldots, x_{2N}; \bs{u}^L)}{\LZ_{\GFF}(H_{\eta}^L; x_{2j+1}, \ldots, x_{2N})}. 
\end{align}
\end{lemma}
\begin{proof}
The law of $\eta$ under $\PP$ is the same as $\SLE_4$ in $(\HH; x_1, \infty)$ weighted by the local martingale $\{M_t(\LZ_{\GFF})\}_t$. 
The law of $\eta$ under $\PP_{\bs{\gamma}}$ is the same as $\SLE_4$ in $(\HH; x_1, \infty)$ weighted by the local martingale $\{M_t(\LZ_{\GFF}^{(\bs{\gamma})})\}_t$, due to the proof of Lemmas~\ref{lem::multiple_log_mart} and~\ref{lem::GFFlog_levelline_coupling}.  
Thus, the law of $\eta$ under $\PP_{\bs{\gamma}}$ is the same as its law under $\PP$ weighted by the local martingale
\begin{align*}
N_t:=&\frac{M_t(\LZ_{\GFF}^{(\bs{\gamma})})}{M_t(\LZ_{\GFF})}
=\prod_{1\le i<k\le M}\exp(\gamma_i\gamma_k\Green(\HH; g_t(u_i), g_t(u_k)))\times \prod_{\ell=1}^M \left(\frac{\CR(\HH; g_t(u_{\ell}))}{|g_t'(u_{\ell})|}\right)^{\frac{\gamma_{\ell}^2}{2}}\exp(\gamma_{\ell}\Gamma_0(\HH; \bs{X}_t; g_t(u_{\ell}))), 
\end{align*}
up to $\sigma$. 
From the conformal covariance of conformal radius~\eqref{eqn::CR_cov}, and the conformal invariance of Green's function~\eqref{eqn::Green_inv} and of the harmonic function $\Gamma_0$, the local martingale $N_t$ can also be written as 
\begin{align}\label{eqn::cascade_singlecurve_aux1}
N_t=&\prod_{1\le i<k\le M}\exp(\gamma_i\gamma_k\Green(\HH\setminus\eta_{[0,t]}; u_i, u_k))\times \prod_{\ell=1}^M \CR(\HH\setminus\eta_{[0,t]}; u_{\ell})^{\frac{\gamma_{\ell}^2}{2}}\exp(\gamma_{\ell}\Gamma_0(\HH\setminus\eta_{[0,t]}; \bs{x}; u_{\ell})).
\end{align}

The level line $\eta$ under $\PP$ is continuous simple curve up to and including $\sigma$ and $\eta_{\sigma}\in \{x_2, x_4, \ldots, x_{2N}\}$. Let us consider the limit of $N_t$ as $t\to \sigma$. 
We have, almost surely under $\PP$, 
\begin{align*}
&\lim_{t\to\sigma}\Green(\HH\setminus\eta_{[0,t]}; u_i, u_k)=\Green(\HH\setminus\eta; u_i, u_k), \qquad && 0\le \Green(\HH\setminus\eta_{[0,t]}; u_i, u_k)\le \Green(\HH; u_i, u_k) \text{ for all }0\le t\le\sigma; \\
&\lim_{t\to\sigma}\CR(\HH\setminus\eta_{[0,t]}; u_{\ell})=\CR(\HH\setminus\eta; u_{\ell}), && 0\le \CR(\HH\setminus\eta_{[0,t]}; u_{\ell})\le \CR(\HH; u_{\ell})\text{ for all }0\le t\le\sigma; \\
&\lim_{t\to\sigma}\Gamma_0(\HH\setminus\eta_{[0,t]}; \bs{x}; u_{\ell})=\Gamma_0(\HH\setminus\eta; \bs{x}; u_{\ell}), && 0\le \Gamma_0(\HH\setminus\eta_{[0,t]}; \bs{x}; u_{\ell})\le \pi\text{ for all }0\le t\le\sigma. 
\end{align*}
Therefore, $\{N_{t\wedge\sigma}\}_t$ is a bounded martingale under $\PP$ and 
\begin{align}\label{eqn::cascade_singlecurve_aux2}
N_{\sigma}=\lim_{t\to\sigma}N_t=&\prod_{1\le i<k\le M}\exp(\gamma_i\gamma_k\Green(\HH\setminus\eta; u_i, u_k))\times \prod_{\ell=1}^M \CR(\HH\setminus\eta; u_{\ell})^{\frac{\gamma_{\ell}^2}{2}}\exp(\gamma_{\ell}\Gamma_0(\HH\setminus\eta; \bs{x}; u_{\ell})).
\end{align}
If $u_i, u_k$ are disconnected by $\eta$, then $\Green(\HH\setminus\eta; u_i, u_k)=0$. Thus, the RHS of~\eqref{eqn::cascade_singlecurve_aux2} is the same as the RHS of~\eqref{eqn::cascade_singlecurve_RN} as desired. 
\end{proof}

\begin{proof}[Proof of Theorem~\ref{thm::multiplelog_polygon}]
We prove the conclusion by induction on $N$. 
When $N=1$, assume the same notation as in Lemma~\ref{lem::cascade_singlecurve}. 
The level line $\eta$ starts from $x_1$ and terminates at $x_2$ at time $\sigma$. 
Lemma~\ref{lem::cascade_singlecurve} tells that the law of $\eta$ under $\PP_{\bs{\gamma}}$ is absolutely continuous with respect to its law under $\PP=\PP_{\bs{0}}$ and the Radon-Nikodym derivative is given by 
\begin{align*}
&\frac{\LZ_{\GFF}(\Omega; x_1, x_2)}{\LZ_{\GFF}^{(\bs{\gamma})}(\Omega; x_1, x_2; \bs{u})}\LZ_{\GFF}^{(\bs{\gamma}^R)}(H_{\eta}^R; \emptyset; \bs{u}^R)\LZ_{\GFF}^{(\bs{\gamma}^L)}(H_{\eta}^L; \emptyset; \bs{u}^L)\\
=&\prod_{1\le i<k\le M}\frac{\exp(\gamma_i\gamma_k \Green(\Omega\setminus\eta; u_i, u_k))}{\exp(\gamma_i\gamma_k \Green(\Omega; u_i, u_k))}\times\prod_{\ell=1}^M\left(\frac{\CR(\Omega\setminus\eta; u_{\ell})}{\CR(\Omega; u_{\ell})}\right)^{\frac{\gamma_{\ell}^2}{2}}\frac{\left(\one\{u_{\ell}\in H_{\eta}^R\}\exp\left(\pi\gamma_{\ell}\right)+\one\{u_{\ell}\in H_{\eta}^L\}\right)}{\exp(\gamma_{\ell}\Gamma_0(\Omega; x_1, x_2; u_{\ell}))},
\end{align*}
as desired. 
\medbreak

Assume the conclusion holds up to $N-1$, let us consider the case for $N$. 
Assume the same notation as in Lemma~\ref{lem::cascade_singlecurve}. 
Let $\eta$ be the level line starting from $x_1$. 
Denote by $\bs{\eta}^R$ (resp. $\bs{\eta}^L$) the collection of level lines in $H_{\eta}^R$ (resp. in $H_{\eta}^L$). 
Note that $\bs{\eta}=\eta\cup \bs{\eta}^R\cup \bs{\eta}^L$. 
On the event $\{\eta(\sigma)=x_{2j}\}$ for $1\le j\le N$, we have the following observations.
\begin{itemize}
\item From Lemma~\ref{lem::cascade_singlecurve}, the law of $\eta$ under $\PP_{\bs{\gamma}}$ is the same as its law under $\PP=\PP_{\bs{0}}$ weighted by 
\[\frac{\LZ_{\GFF}(\Omega; \bs{x})}{\LZ_{\GFF}^{(\bs{\gamma})}(\Omega; \bs{x}; \bs{u})}\frac{\LZ_{\GFF}^{(\bs{\gamma}^R)}(H_{\eta}^R; x_3, \ldots, x_{2j-1}, x_2; \bs{u}^R)}{\LZ_{\GFF}(H_{\eta}^R; x_3, \ldots, x_{2j-1}, x_2)}\frac{\LZ_{\GFF}^{(\bs{\gamma}^L)}(H_{\eta}^L; x_{2j+1}, \ldots, x_{2N}; \bs{u}^L)}{\LZ_{\GFF}(H_{\eta}^L; x_{2j+1}, \ldots, x_{2N})}. 
\]
\item From the domain Markov property of the GFF and the induction hypothesis, given $\eta$, the conditional law of $\bs{\eta}^L$ under $\PP_{\bs{\gamma}}$ is the same as its law under $\PP$ weighted by 
\begin{align*}
&\prod_{\substack{i<k\\ u_i, u_k\in\bs{u}^L}}\frac{\exp(\gamma_i\gamma_k\Green(\Omega\setminus\bs{\eta}; u_i, u_k))}{\exp(\gamma_i\gamma_k\Green(H_{\eta}^L; u_i, u_k))}\times\prod_{\ell: u_{\ell}\in \bs{u}^L}\left(\frac{\CR(\Omega\setminus\bs{\eta}; u_{\ell})}{\CR(H_{\eta}^L; u_{\ell})}\right)^{\frac{\gamma_{\ell}^2}{2}}\frac{\left(\one\{u_{\ell}\in \Omega_{\bs{\eta}}^+\}\exp(\pi\gamma_{\ell})+\one\{u_{\ell}\in \Omega_{\bs{\eta}}^-\}\right)}{\exp(\gamma_{\ell}\Gamma_0(H_{\eta}^L; x_{2j+1}, \ldots, x_{2N}; u_{\ell}))}. 
\end{align*}
\item From the domain Markov property of the GFF and the induction hypothesis, given $\eta$, the conditional law of $\bs{\eta}^R$ is the same as the law of $\bs{\gamma}^R$ weighted by 
\begin{align*}
&\prod_{\substack{i<k\\ u_i, u_k\in\bs{u}^R}}\frac{\exp(\gamma_i\gamma_k\Green(\Omega\setminus\bs{\eta}; u_i, u_k))}{\exp(\gamma_i\gamma_k\Green(H_{\eta}^R; u_i, u_k))}\times\prod_{\ell: u_{\ell}\in \bs{u}^R}\left(\frac{\CR(\Omega\setminus\bs{\eta}; u_{\ell})}{\CR(H_{\eta}^R; u_{\ell})}\right)^{\frac{\gamma_{\ell}^2}{2}}\frac{\left(\one\{u_{\ell}\in \Omega_{\bs{\eta}}^+\}\exp(\pi\gamma_{\ell})+\one\{u_{\ell}\in \Omega_{\bs{\eta}}^-\}\right)}{\exp(\gamma_{\ell}\Gamma_0(H_{\eta}^R; x_3, \ldots, x_{2j-1}, x_2; u_{\ell}))}. 
\end{align*}
\end{itemize}
Combining the three Radon-Nikodym derivatives, the law of $\bs{\eta}$ under $\PP_{\bs{\gamma}}$ on the event $\{\eta(\sigma)=x_{2j}\}$ is the same as its law under $\PP$ on the same event weighted by~\eqref{eqn::multiplelog_polygon_RN}. This completes the proof. 
\end{proof}

\subsection{Proof of Proposition~\ref{prop::onelog_CR_fourier}}
\label{subsec::onelog_CR_fourier}
\begin{lemma}
Assume the same notation as in Theorem~\ref{thm::multiplelog_polygon}  with $M=1$.
Fix $u_1=u\in\Omega$ and $\gamma_1=\gamma\in\R$. We have
\begin{align}\label{eqn::onelog_CR_negexponent_aux}
\begin{split}
\E\left[\left(\frac{\CR(\Omega\setminus\bs{\eta}; u)}{\CR(\Omega; u)}\right)^{\frac{\gamma^2}{2}}\one\{u\in\Omega_{\bs{\eta}}^+\}\right]
=&\frac{\sinh\left(\gamma\Gamma_0(u)\right)}{\sinh(\pi\gamma)}. 
\end{split}
\end{align}
Moreover, 
\begin{align}\label{eqn::onelog_proba}
\PP_{\gamma}\left[u\in\Omega_{\bs{\eta}}^+\right]=\frac{1-\exp(-2\gamma\Gamma_0(u))}{1-\exp(-2\pi\gamma)}, \quad\text{for }\gamma\neq 0; \qquad \PP\left[u\in\Omega_{\bs{\eta}}^+\right]=\frac{\Gamma_0(u)}{\pi}, \quad\text{for }\gamma=0. 
\end{align}

\end{lemma}
\begin{proof}
The law of $\bs{\eta}$ under $\PP_{\gamma}$ is the same as its law under $\PP$ weighted by the Radon-Nikodym derivative given by 
\begin{align*}
\left(\frac{\CR(\Omega\setminus\bs{\eta}; u)}{\CR(\Omega; u)}\right)^{\frac{\gamma^2}{2}}\frac{\left(\one\{u\in\Omega_{\bs{\eta}}^+\}\exp\left(\pi\gamma\right)+\one\{u\in\Omega_{\bs{\eta}}^-\}\right)}{\exp(\gamma\Gamma_0(u))}. 
\end{align*}
Taking the expectation of the Radon-Nikodym derivative, we have
\begin{align}\label{eqn::onelog_cr_aux1}
\begin{split}
&\E\left[\left(\frac{\CR(\Omega\setminus\bs{\eta}; u)}{\CR(\Omega; u)}\right)^{\frac{\gamma^2}{2}}\one\{u\in\Omega_{\bs{\eta}}^+\}\right]\exp\left(\pi\gamma\right)
+\E\left[\left(\frac{\CR(\Omega\setminus\bs{\eta}; u)}{\CR(\Omega; u)}\right)^{\frac{\gamma^2}{2}}\one\{u\in\Omega_{\bs{\eta}}^-\}\right]
=\exp(\gamma\Gamma_0(u)). 
\end{split}
\end{align}
Replacing $\gamma$ by $-\gamma$, we have
\begin{align}\label{eqn::onelog_cr_aux2}
\begin{split}
&\E\left[\left(\frac{\CR(\Omega\setminus\bs{\eta}; u)}{\CR(\Omega; u)}\right)^{\frac{\gamma^2}{2}}\one\{u\in\Omega_{\bs{\eta}}^+\}\right]\exp\left(-\pi\gamma\right)
+\E\left[\left(\frac{\CR(\Omega\setminus\bs{\eta}; u)}{\CR(\Omega; u)}\right)^{\frac{\gamma^2}{2}}\one\{u\in\Omega_{\bs{\eta}}^-\}\right]
=\exp(-\gamma\Gamma_0(u)). 
\end{split}
\end{align}
Solving the system~\eqref{eqn::onelog_cr_aux1}-\eqref{eqn::onelog_cr_aux2}, we obtain~\eqref{eqn::onelog_CR_negexponent_aux} as desired. 

Furthermore, we have
\begin{align*}
\PP_{\gamma}\left[u\in\Omega_{\bs{\eta}}^+\right]=&\E\left[\left(\frac{\CR(\Omega\setminus\bs{\eta}; u)}{\CR(\Omega; u)}\right)^{\frac{\gamma^2}{2}}\frac{\one\{u\in\Omega_{\bs{\eta}}^+\}\exp\left(\pi\gamma\right)}{\exp(\gamma\Gamma_0(u))}\right]\\
=&\frac{\sinh\left(\gamma\Gamma_0(u)\right)}{\sinh(\pi\gamma)}\frac{\exp\left(\pi\gamma\right)}{\exp(\gamma\Gamma_0(u))},\tag{due to~\eqref{eqn::onelog_CR_negexponent_aux}}
\end{align*}
as desired in~\eqref{eqn::onelog_proba}. 
\end{proof}

\begin{corollary}\label{cor::onelog_extension}
Assume the same notation as in Theorem~\ref{thm::multiplelog_polygon}  with $M=1$.
Fix $u_1=u\in\Omega$. Then for all $\xi\in\C$ with $\Re{\xi}<1/2$, we have
\begin{align}\label{eqn::onelog_CR_fourier_aux1}
\E\left[\left(\frac{\CR(\Omega; u)}{\CR(\Omega\setminus\bs{\eta}; u)}\right)^{\xi}\one\{u\in\Omega_{\bs{\eta}}^+\}\right]
=&\frac{\sin\left(\Gamma_0(u)\sqrt{2\xi}\right)}{\sin(\pi\sqrt{2\xi})}. 
\end{align}
\end{corollary}
\begin{proof}
Combining~\cite[Theorem~1.4]{PeltolaWuGlobalMultipleSLEs} and~\cite[Lemma~2.1]{FengWuRadialBPZFKIsing} with $\kappa=4$, for $\beta<1/2$, there exist a constant $C_{\beta}\in (0,\infty)$ depending only on $\beta$ and a function $\LB(\Omega; \bs{x})\in (0,\infty)$ depending only on the polygon $(\Omega; \bs{x})$ such that 
\begin{align}\label{eqn::onelog_CR_posexponent_aux1}
\E\left[\left(\frac{\CR(\Omega; u)}{\CR(\Omega\setminus\bs{\eta}; u)}\right)^{\beta}\right]\le C_{\beta}\LB(\Omega; \bs{x}), \qquad\text{for all }\beta<\frac{1}{2}. 
\end{align}

We define 
\begin{align*}
g(\xi; \bs{\eta}):=\left(\frac{\CR(\Omega; u)}{\CR(\Omega\setminus\bs{\eta}; u)}\right)^{\xi}, \qquad F(\xi):=\E[g(\xi; \bs{\eta})\one\{u\in\Omega_{\bs{\eta}}^+\}], \qquad\text{for }\xi\in\C\text{ with }\Re{\xi}<\frac{1}{2}. 
\end{align*}
The expectation $\E[g(\xi; \bs{\eta})\one\{u\in\Omega_{\bs{\eta}}^+\}]$ is finite due to~\eqref{eqn::onelog_CR_posexponent_aux1}. 
Eq.~\eqref{eqn::onelog_CR_fourier_aux1} follows from the following two observations.
\begin{itemize}
\item Eq.~\eqref{eqn::onelog_CR_fourier_aux1} holds for $\xi=\beta$ is real and $\beta\le 0$. Note that
\begin{align}\label{eqn::onelog_CR_posexponent_aux2}
F(\beta)=\frac{\sinh(\Gamma_0(u)\sqrt{-2\beta})}{\sinh(\pi\sqrt{-2\beta})} 
\end{align}
by setting $\gamma=\sqrt{-2\beta}$ in~\eqref{eqn::onelog_CR_negexponent_aux}. The RHS of~\eqref{eqn::onelog_CR_posexponent_aux2} is the same as the RHS of~\eqref{eqn::onelog_CR_fourier_aux1} by noting that $\ii\sin( z)=\sinh(\ii z)$.
\item The function $\xi\mapsto F(\xi)$ is analytic for $\xi\in\C$ with $\Re{\xi}<1/2$. Note that $\xi\mapsto g(\xi; \bs{\eta})\one\{u\in\Omega_{\bs{\eta}}^+\}$ is analytic almost surely. Fix $\beta$ real with $\beta<1/2$, from~\eqref{eqn::onelog_CR_posexponent_aux1}, we have $|g(\xi; \bs{\eta})|\le g(\beta; \bs{\eta})$ and $\E[g(\beta; \bs{\eta})]\le C_{\beta}\LB(\Omega; \bs{x})$ for all $\xi\in\C$ with $\Re{\xi}\le \beta$. This confirms that $F(\xi)=\E[g(\xi; \bs{\eta})\one\{u\in\Omega_{\bs{\eta}}^+\}]$ is also analytic for $\xi\in\C$ with $\Re{\xi}<\beta$ as desired. 
\end{itemize}
\end{proof}

\begin{proof}[Proof of Proposition~\ref{prop::onelog_CR_fourier}]
For $\xi\in\C$ with $\Re{\xi}<(1+\gamma^2)/2$, from Theorem~\ref{thm::multiplelog_polygon}, we have
\begin{align*}
\E_{\gamma}\left[\left(\frac{\CR(\Omega; u)}{\CR(\Omega\setminus\bs{\eta}; u)}\right)^{\xi}\one\{u\in\Omega_{\bs{\eta}}^+\}\right]
=&\exp(\gamma(\pi-\Gamma_0(u)))\E\left[\left(\frac{\CR(\Omega; u)}{\CR(\Omega\setminus\bs{\eta}; u)}\right)^{\xi-\frac{\gamma^2}{2}}\one\{u\in\Omega_{\bs{\eta}}^+\}\right]\\
=&\exp(\gamma(\pi-\Gamma_0(u)))\frac{\sin\left(\Gamma_0(u)\sqrt{2\xi-\gamma^2}\right)}{\sin\left(\pi\sqrt{2\xi-\gamma^2}\right)}.\tag{due to~\eqref{eqn::onelog_CR_fourier_aux1}}
\end{align*}
This completes the proof of~\eqref{eqn::onelog_CR_fourier}. 
\end{proof}

\subsection{Proof of Proposition~\ref{prop::twopoint_correlation}}
\label{subsec::twopoint_correlation}

\begin{lemma}
Assume the same notation as in Theorem~\ref{thm::multiplelog_polygon}  with $M=2$. 
Fix distinct $u_1, u_2\in\Omega$ and $\gamma\in\R$, we have 
\begin{align}\label{eqn::twopoints_neg}
\begin{split}
&\E\left[\left(\frac{\CR(\Omega\setminus\bs{\eta}; u_1)}{\CR(\Omega; u_1)}\right)^{\frac{\gamma^2}{2}}\left(\frac{\CR(\Omega\setminus\bs{\eta}; u_2)}{\CR(\Omega; u_2)}\right)^{\frac{\gamma^2}{2}}\left(\one\{u_1\in\Omega_{\bs{\eta}}^+\}-\one\{u_2\in\Omega_{\bs{\eta}}^+\}\right)\right]\\
=&\exp(-\gamma^2\Green(\Omega; u_1, u_2))\frac{\sinh(\gamma(\Gamma_0( u_1)-\Gamma_0( u_2)))}{\sinh(\pi\gamma)}. 
\end{split}
\end{align}
\end{lemma}

\begin{proof}
The law of $\bs{\eta}$ under $\PP_{\bs{\gamma}}$ is the same as its law under $\PP$ weighted by the Radon-Nikodym derivative given by
\begin{align*}
\frac{\exp(\gamma_1\gamma_2\Green(\Omega\setminus\bs{\eta}; u_1, u_2))}{\exp(\gamma_1\gamma_2\Green(\Omega; u_1, u_2))}
\times \prod_{j=1}^2\left(\frac{\CR(\Omega\setminus\bs{\eta}; u_{j})}{\CR(\Omega; u_{j})}\right)^{\frac{\gamma_{j}^2}{2}}\frac{\left(\one\{u_{j}\in\Omega_{\bs{\eta}}^+\}\exp\left(\pi\gamma_{j}\right)+\one\{u_{j}\in\Omega_{\bs{\eta}}^-\}\right)}{\exp(\gamma_{j}\Gamma_0(u_{j}))}. 
\end{align*}
Consider the following four events:
\begin{align*}
A=\{u_1, u_2\in\Omega_{\bs{\eta}}^+\}, \qquad B=\{u_1\in\Omega_{\bs{\eta}}^+, u_2\in\Omega_{\bs{\eta}}^-\}, \qquad C=\{u_1\in\Omega_{\bs{\eta}}^-, u_2\in\Omega_{\bs{\eta}}^+\}, \qquad D=\{u_1, u_2\in\Omega_{\bs{\eta}}^-\}. 
\end{align*}
Taking the expectation of the Radon-Nikodym derivative, we have
\begin{align}\label{eqn::twopoints_aux1}
\begin{split}
&\exp(\gamma_1\gamma_2\Green(\Omega; u_1, u_2))\exp(\gamma_1\Gamma_0(u_1)+\gamma_2\Gamma_0(u_2))\\
=&\E\left[\one_A\exp(\gamma_1\gamma_2\Green(\Omega\setminus\bs{\eta}; u_1, u_2))
\times \prod_{j=1}^2\left(\frac{\CR(\Omega\setminus\bs{\eta}; u_{j})}{\CR(\Omega; u_{j})}\right)^{\frac{\gamma_{j}^2}{2}}\right]\exp(\pi(\gamma_1+\gamma_2))\\
&+\E\left[\one_B \prod_{j=1}^2\left(\frac{\CR(\Omega\setminus\bs{\eta}; u_{j})}{\CR(\Omega; u_{j})}\right)^{\frac{\gamma_{j}^2}{2}}\right]\exp(\pi\gamma_1)+\E\left[\one_C \prod_{j=1}^2\left(\frac{\CR(\Omega\setminus\bs{\eta}; u_{j})}{\CR(\Omega; u_{j})}\right)^{\frac{\gamma_{j}^2}{2}}\right]\exp(\pi\gamma_2)\\
&+\E\left[\one_D\exp(\gamma_1\gamma_2\Green(\Omega\setminus\bs{\eta}; u_1, u_2))
\times \prod_{j=1}^2\left(\frac{\CR(\Omega\setminus\bs{\eta}; u_{j})}{\CR(\Omega; u_{j})}\right)^{\frac{\gamma_{j}^2}{2}}\right], 
\end{split}
\end{align}
where we use the fact that $\Green(\Omega\setminus\bs{\eta}; u_1, u_2)=0$ on the events $B$ and $C$. 
Setting $\gamma_1=-\gamma_2=\gamma$ in~\eqref{eqn::twopoints_aux1}, we obtain 
\begin{align}\label{eqn::twopoints_aux2}
\begin{split}
&\exp(-\gamma^2\Green(\Omega; u_1, u_2))\exp(\gamma\Gamma_0(u_1)-\gamma\Gamma_0(u_2))\\
=&\E\left[\one_A\exp(-\gamma^2\Green(\Omega\setminus\bs{\eta}; u_1, u_2))
\times \prod_{j=1}^2\left(\frac{\CR(\Omega\setminus\bs{\eta}; u_{j})}{\CR(\Omega; u_{j})}\right)^{\frac{\gamma_{j}^2}{2}}\right]\\
&+\E\left[\one_B \prod_{j=1}^2\left(\frac{\CR(\Omega\setminus\bs{\eta}; u_{j})}{\CR(\Omega; u_{j})}\right)^{\frac{\gamma_{j}^2}{2}}\right]\exp(\pi\gamma)+\E\left[\one_C \prod_{j=1}^2\left(\frac{\CR(\Omega\setminus\bs{\eta}; u_{j})}{\CR(\Omega; u_{j})}\right)^{\frac{\gamma_{j}^2}{2}}\right]\exp(-\pi\gamma)\\
&+\E\left[\one_D\exp(-\gamma^2\Green(\Omega\setminus\bs{\eta}; u_1, u_2))
\times \prod_{j=1}^2\left(\frac{\CR(\Omega\setminus\bs{\eta}; u_{j})}{\CR(\Omega; u_{j})}\right)^{\frac{\gamma_{j}^2}{2}}\right].
\end{split}
\end{align}

Setting $\gamma_1=-\gamma_2=-\gamma$ in~\eqref{eqn::twopoints_aux1}, we obtain 
\begin{align}\label{eqn::twopoints_aux3}
\begin{split}
&\exp(-\gamma^2\Green(\Omega; u_1, u_2))\exp(-\gamma\Gamma_0(u_1)+\gamma\Gamma_0(u_2))\\
=&\E\left[\one_A\exp(-\gamma^2\Green(\Omega\setminus\bs{\eta}; u_1, u_2))
\times \prod_{j=1}^2\left(\frac{\CR(\Omega\setminus\bs{\eta}; u_{j})}{\CR(\Omega; u_{j})}\right)^{\frac{\gamma_{j}^2}{2}}\right]\\
&+\E\left[\one_B \prod_{j=1}^2\left(\frac{\CR(\Omega\setminus\bs{\eta}; u_{j})}{\CR(\Omega; u_{j})}\right)^{\frac{\gamma_{j}^2}{2}}\right]\exp(-\pi\gamma)+\E\left[\one_C \prod_{j=1}^2\left(\frac{\CR(\Omega\setminus\bs{\eta}; u_{j})}{\CR(\Omega; u_{j})}\right)^{\frac{\gamma_{j}^2}{2}}\right]\exp(\pi\gamma)\\
&+\E\left[\one_D\exp(-\gamma^2\Green(\Omega\setminus\bs{\eta}; u_1, u_2))
\times \prod_{j=1}^2\left(\frac{\CR(\Omega\setminus\bs{\eta}; u_{j})}{\CR(\Omega; u_{j})}\right)^{\frac{\gamma_{j}^2}{2}}\right]. 
\end{split}
\end{align}
Taking the difference between~\eqref{eqn::twopoints_aux2} and~\eqref{eqn::twopoints_aux3}, we obtain~\eqref{eqn::twopoints_neg} as desired.
\end{proof}

\begin{corollary}
Assume the same notation as in Theorem~\ref{thm::multiplelog_polygon}  with $M=2$. 
Fix distinct $u_1, u_2\in\Omega$. 
Then for $\xi\in\C$ with $\Re{\xi}\le 0$, we have 
\begin{align}\label{eqn::twopoints_extension}
\begin{split}
&\E\left[\left(\frac{\CR(\Omega; u_1)}{\CR(\Omega\setminus\bs{\eta}; u_1)}\right)^{\xi}\left(\frac{\CR(\Omega; u_2)}{\CR(\Omega\setminus\bs{\eta}; u_2)}\right)^{\xi}\left(\one\{u_1\in\Omega_{\bs{\eta}}^+\}-\one\{u_2\in\Omega_{\bs{\eta}}^+\}\right)\right]\\
=&\exp(2\xi\Green(\Omega; u_1, u_2))\frac{\sin((\Gamma_0( u_1)-\Gamma_0( u_2))\sqrt{2\xi})}{\sin(\pi\sqrt{2\xi})}. 
\end{split}
\end{align}
\end{corollary}
\begin{proof}
This can be proved using a similar argument as in the proof of Corollary~\ref{cor::onelog_extension}. Note that we no longer have analogous conclusion of~\eqref{eqn::onelog_CR_posexponent_aux1} in this case, but we have $0\le \CR(\Omega\setminus\bs{\eta}; u_j)\le \CR(\Omega; u_j)$. Thus the proof for the analytic extension in this case only works for $\Re{\xi}\le 0$. 
\end{proof}

\begin{proof}[Proof of Proposition~\ref{prop::twopoint_correlation}]
For $\xi\in\C$ with $\Re{\xi}\le \gamma^2/2$, from Theorem~\ref{thm::multiplelog_polygon}, we have
\begin{align*}
&\E_{\bs{\gamma}}\left[\left(\frac{\CR(\Omega; u_1)}{\CR(\Omega\setminus\bs{\eta}; u_1)}\right)^{\xi}\left(\frac{\CR(\Omega; u_2)}{\CR(\Omega\setminus\bs{\eta}; u_2)}\right)^{\xi}\left(\one\{u_1\in\Omega_{\bs{\eta}}^+\}-\one\{u_2\in\Omega_{\bs{\eta}}^+\}\right)\right]\\
=&\E_{\bs{\gamma}}\left[\left(\frac{\CR(\Omega; u_1)}{\CR(\Omega\setminus\bs{\eta}; u_1)}\right)^{\xi}\left(\frac{\CR(\Omega; u_2)}{\CR(\Omega\setminus\bs{\eta}; u_2)}\right)^{\xi}\left(\one\{u_1\in\Omega_{\bs{\eta}}^+, u_2\in\Omega_{\bs{\eta}}^-\}-\one\{u_1\in\Omega_{\bs{\eta}}^-, u_2\in\Omega_{\bs{\eta}}^+\}\right)\right]\\
=&\E\left[\left(\frac{\CR(\Omega; u_1)}{\CR(\Omega\setminus\bs{\eta}; u_1)}\right)^{\xi-\frac{\gamma^2}{2}}\left(\frac{\CR(\Omega; u_2)}{\CR(\Omega\setminus\bs{\eta}; u_2)}\right)^{\xi-\frac{\gamma^2}{2}}\left(\one\{u_1\in\Omega_{\bs{\eta}}^+, u_2\in\Omega_{\bs{\eta}}^-\}-\one\{u_1\in\Omega_{\bs{\eta}}^-, u_2\in\Omega_{\bs{\eta}}^+\}\right)\right]\\
&\times\frac{\exp(\pi\gamma)}{\exp(\gamma^2\Green(\Omega; u_1, u_2)+\gamma(\Gamma_0(u_1)+\Gamma_0(u_2)))}\\
=&\exp((2\xi-\gamma^2)\Green(\Omega; u_1, u_2))\frac{\sin\left((\Gamma_0( u_1)-\Gamma_0( u_2))\sqrt{2\xi-\gamma^2}\right)}{\sin\left(\pi\sqrt{2\xi-\gamma^2}\right)}\frac{\exp(\pi\gamma)}{\exp(\gamma^2\Green(\Omega; u_1, u_2)+\gamma(\Gamma_0(u_1)+\Gamma_0(u_2)))},
\end{align*}
where the last equal sign is due to~\eqref{eqn::twopoints_extension}. This completes the proof of~\eqref{eqn::twopoint_correlation}. 
\end{proof}


\section{Brownian motion representation}
\label{sec::Brownian_represent}
We will prove Theorem~\ref{thm::Brownian_represent} in this section. 
The idea is to replace the punctures $(u_\ell)_{1\le\ell\le M}$
by macroscopic holes $\cup_{1\le\ell\le M}B(u_\ell,\eps)$,
and use \cite{AruLupuSepulvedaFPSGFFCVGISO} for the representation of level lines in a multiply connected domain as envelops of clusters of Brownian paths. Then one lets $\eps\to 0$ and obtains Theorem~\ref{thm::Brownian_represent}
through the convergence in total variation of level lines.

\subsection{Conclusion with cutoff}
Define the multiply-connected domain $\Omega_\eps:=\Omega\setminus\left(\cup_{1\le\ell\le M}B(u_\ell,\eps)\right)$ for $\eps>0$. Suppose $\Gamma_\eps$ is $\GFF$ in $\Omega_\eps$ and the boundary data $h_{\eps}$ is given as follows: $h_{\eps}$ equals $\pi$ on $\partial^+\Omega$ and it equals $0$ on $\partial^-\Omega$, and
\begin{equation}\label{eqn::multiplelog_polygon_boundarydata2}
h_{\eps}(z)=\sum_{1\le\ell\le M}\gamma_\ell \Green(\Omega;z,u_\ell) \qquad\text{for }z\in\cup_{1\le\ell\le M}\partial B(u_\ell,\eps).
\end{equation}
\begin{proposition}\label{prop::BES_truncated}
Assume the the same setup as in Theorem~\ref{thm::Brownian_represent}. We sample $2+\frac{1}{2}M(M+3)$ independent Poisson point processes. 
\begin{enumerate}
\item[(1)] Poisson point process of Brownian loops inside $\Omega_\eps$ with intensity $c=1/2$.
\item[(2)] Poisson point process of Brownian excursions with intensity
\begin{align}
\frac{\pi}{4}\iint_{\partial^+\Omega\times \partial^+\Omega}\ud x\ud y \;\Poisson(\Omega_\eps; x, y)\bm^{\#}(\Omega_\eps; x, y)[\cdot].
\end{align}
\item[(3)] For each $u_{\ell}$, Poisson point process of Brownian excursions in $\Omega_{\eps}$ from $\partial B(u_\ell,\eps)$ to $\partial B(u_\ell,\eps)$ with intensity
\begin{align}
\frac{1}{4\pi}\iint_{\partial B(u_\ell,\eps)\times \partial B(u_\ell,\eps)}\ud x\ud y \;h_{\eps}(x)h_{\eps}(y)\Poisson(\Omega_\eps; x, y)\bm^{\#}(\Omega_\eps; x, y)[\cdot]. 
\end{align}
\item[(4)] For each $u_{\ell}$, Poisson point process of Brownian excursions in $\Omega_\eps$ from $\partial B(u_\ell,\eps)$ to $\partial^+\Omega$ with intensity
\begin{align}
\frac{1}{2}\iint_{\partial B(u_\ell,\eps)\times\partial^+\Omega}\ud x\ud y \;h_{\eps}(x)\Poisson(\Omega_\eps; x, y)\bm^{\#}(\Omega_\eps; x, y)[\cdot].
\end{align}
\item[(5)] For each pair $u_i\neq u_j$, Poisson point process of Brownian excursions from $\partial B(u_i,\eps)$ to $\partial B(u_j,\eps)$ with intensity
\begin{align}
\frac{1}{2\pi}\iint_{\partial B(u_i,\eps)\times \partial B(u_j,\eps)}\ud x\ud y \;h_{\eps}(x)h_{\eps}(y)\Poisson(\Omega_\eps; x, y)\bm^{\#}(\Omega_\eps; x, y)[\cdot].
\end{align}
\end{enumerate}
We consider the clusters formed by the collection of all the above Brownian trajectories,
and among these clusters, specifically the clusters connected to the outer boundary $\partial\Omega$,
which are exactly the clusters containing at least one excursion from (2) or (4).
Let $\bs{\eta}_{\eps}$ be the envelop of these clusters connected to $\partial\Omega$, which consists of
$N$ disjoint curves in $\Omega_\eps$ forming a pairing of $\{x_1,\dots ,x_{2N}\}$. Then $\bs{\eta}_{\eps}$ has the same law as the union of the generalized level lines (see~\cite{AruSepulvedaWernerBTLStwoDGFF}) of $\Gamma_\eps$ starting from $x_{2j-1}$ for $1\le j\le N$.
\end{proposition}
\begin{proof}
See~\cite[Corollary 5.10]{AruLupuSepulvedaFPSGFFCVGISO} for the case $N=1$. For the general case $N\ge 2$, we can use almost the same proof of~\cite[Lemma 5.4]{LiuWuCrossingProbaMGFF}. The proof consists of two steps. For completeness, we summarize the strategy briefly below.
\begin{itemize}
\item
Given $\bs{\eta}_{\eps}$, denote by $\LC_\eps$ the union of connected components of $\Omega_\eps\setminus\bs{\eta}_{\eps}$ wohse boundaries contain $\partial^+\Omega$. By the property of generlized level lines, the conditional law of $\Gamma_\eps$ equals $\GFF$ in $\LC_\eps$ and the boundary data $v_\eps$ is given as follows: $v_\eps=\pi$ on $\partial\LC_\eps\setminus \left(\cup_{1\le\ell\le M}\partial B(u_\ell,\eps)\right)$ and $v_\eps=h_{\eps}$ on $\cup_{1\le\ell\le M}\partial B(u_\ell,\eps)$. Sample the first passage set $\A^{v_\eps}_0$ in $\LC_\eps$ (see~\cite[Definition 3.3]{AruLupuSepulvedaFPSGFFCVGISO} for the definition of the first passage set). By the same proof in~\cite[Lemma 5.4]{LiuWuCrossingProbaMGFF}, we can prove that $\A^{v_\eps}_0\cup\partial\Omega$ is the first passage set $\A^{h_{\eps}}_0$ of $\Gamma_\eps$ in $\Omega$. In particular, we have that $\bs{\eta}_{\eps}$ is the frontier of $\A^{h_{\eps}}_0$.
\item
By~\cite[Proposition 5.3]{AruLupuSepulvedaFPSGFFCVGISO}, the first passage set $\A^{h_{\eps}}_0$ has the same law as the cluster of Brownian trajectories. Note that in~\cite[Proposition 5.3]{AruLupuSepulvedaFPSGFFCVGISO}, the authors required the harmonic function is piecewise constant on the boundary, but the same proof works in the setting of Proposition~\ref{prop::BES_truncated}.
\end{itemize}
By the above argument, we complete the proof.
\end{proof}

\subsection{Convergence in total variation}
In this section, we compare the Brownian excursion soups in Proposition~\ref{prop::BES_truncated} and the Brownian excursion soups in Theorem~\ref{thm::Brownian_represent}. Denote by $\bs{\eta}$ the Brownian paths in Theorem~\ref{thm::Brownian_represent}. 
First, we consider a Brownian loop soup with intensity $1/2$ in the whole domain $\Omega$,
and denote by $K_{\rm ext}(\bs{u})$ the union of the topological closures of outermost clusters of the loop soup that surround at least one of the points $u_{\ell}$. 
By construction,  $K_{\rm ext}(\bs{u})$ is a random closed subset of $\Omega$.
Moreover, a.s., 
$K_{\rm ext}(\bs{u})\cap\{u_1,\dots u_M\} = \emptyset$.
For $\delta>0$, we denote by $E_\delta$ the event
\begin{displaymath}
E_\delta = \{K_{\rm ext}(\bs{u})\subset \Omega_\delta\}.
\end{displaymath}
By the above,
\begin{displaymath}
\lim_{\delta\to 0}\PP[E_\delta] = 1.
\end{displaymath}

Further, we take $\delta>\eps>0$. A Brownian loop soup in $\Omega_\eps$ can be obtained as a subfamily of a Brownian loop soup in $\Omega$.
Next, we consider the families of trajectories (2), (3), (4) and (5) from Theorem \ref{thm::Brownian_represent}.
From these, we remove the trajectories contained in $\Omega\setminus \Omega_\delta$ (in case of (2)).
As for the trajectories intersecting $\Omega_\delta$, we remove the initial and the final the stretches contained in 
$\Omega\setminus \Omega_\delta$: no stretches removed from trajectories in (2), 
initial and final stretches removed from trajectories in (3) and (5), 
and initial stretch removed from trajectories in (4).
We will denote by $\Xi_{\delta}$ the obtained family of partial trajectories.
Similarly, we consider the families of trajectories (2), (3), (4) and (5) from
Proposition \ref{prop::BES_truncated}, remove the trajectories contained in $\Omega\setminus \Omega_\delta$
and for the other remove the initial and final stretches if contained in 
$\Omega\setminus \Omega_\delta$.
We will denote by $\Xi_{\delta, \eps}$ the obtained family of partial trajectories.
The goal of the above ``surgery" is to remove the information on whether an excursion starts (and ends) in a interior marked point $u_\ell$ or rather on a small circle around this point.

\begin{lemma}\label{lem::abs_cont_exc}
For $\delta>\eps>0$, the law of $\Xi_{\delta, \eps}$  is absolutely continuous with respect to that of $\Xi_{\delta}$.
Moreover, for fixed $\delta>0$, as $\eps\to 0$, the law of $\Xi_{\delta, \eps}$ converges in total variation to
$\Xi_{\delta}$.
\end{lemma}
\begin{proof}
We will only some details in the case of trajectories of type (4), the other cases being similar.
First, we need an indicatrix that all the excursions are contained in $\Omega_\eps$.
Then fix $\ell\in\{1,\dots,M\}$.
For each erased initial stretch from $u_\ell$ to $y\in \partial B(u_\ell,\delta)$,
we add as a factor
\begin{equation}
\label{Eq int H eps delta}
\dfrac{1}{\gamma_\ell \Poisson(B(u_\ell,\delta);u_\ell,y)}
\int_{\partial B(u_\ell,\eps)} \ud x \, h_\eps(x)\Poisson(B(u_\ell,\delta)\setminus \overline{B(u_\ell,\eps)}; x,y) ,
\end{equation}
where $\Poisson(B(u_\ell,\delta)\setminus \overline{B(u_\ell,\eps)}; \cdot, \cdot)$
is the boundary Poisson kernel in the annulus $B(u_\ell,\delta)\setminus \overline{B(u_\ell,\eps)}$.
Finally, there is a ratio of partition functions which normalizes the density to integral $1$,
which equals
\begin{displaymath}
\exp\Big(\dfrac{\gamma_\ell}{2}\Big(\frac{\Gamma_0(u_\ell)}{\pi} 
- \iint_{\partial B(u_\ell,\eps)\times\partial^{+}\Omega}
\ud x\, \ud z\, h_\eps(x)\Poisson(\Omega_\eps; x,z)
\Big)\Big).
\end{displaymath}
This provides a description of the Radon-Nikodym derivative.
The convergence in total variation follows from the convergence of
$\Poisson((\Omega\setminus \overline{\Omega_\delta})\cap \Omega_\eps;\cdot,\cdot)$
\eqref{Eq int H eps delta} to $1$, and the convergence of
\begin{displaymath}
\pi\iint_{\partial B(u_\ell,\eps)\times\partial^{+}\Omega}
\ud x\, \ud z\, h_\eps(x)\Poisson(\Omega_\eps; x,z)
\end{displaymath}
to $\Gamma_0(u_\ell)$.
\end{proof}

\begin{lemma}\label{lem::conv_envelop_1}
Recall that we denote by $\bs{\eta}_{\eps}$ be the envelop in Proposition~\ref{prop::BES_truncated}. 
As $\eps\to 0$, the law of $\bs{\eta}_{\eps}$ converges in total variation to that of the envelop $\bs{\eta}$
from Theorem~\ref{thm::Brownian_represent}.
\end{lemma}
\begin{proof}
We will denote by $\LL$ the Brownian loop soup in $\Omega$.
Fix $\delta>0$ small enough. 
First, we claim that on the event $E_\delta$ (which depends only on $\LL$),
the envelop $\bs{\eta}$ is measurable with respect to only
$(\LL,\Xi_{\delta})$, i.e. with excursions inside $\Omega\setminus \Omega_\delta$ forgotten.
Indeed, for the construction of $\bs{\eta}$, only the outermost clusters of $\LL$ matter,
and on the event $E_\delta$, no outermost cluster of $\LL$ intersects $\Omega\setminus \Omega_\delta$.
Similarly, for $\eps\in (0,\delta)$, on the event $E_\delta$,
$\bs{\eta}_{\eps}$ is measurable with respect to only
$(\LL,\Xi_{\delta,\eps})$.
Then, the convergence in total variation of $\Xi_{\delta,\eps}$ to $\Xi_{\delta}$
(Lemma \ref{lem::abs_cont_exc}),
ensures that the $\limsup$ of the total variation distance between $\bs{\eta}_{\eps}$ and $\bs{\eta}$
is at most $2(1-\PP(E_\delta))$.
This in turn tends to $0$ as $\delta\to 0$.
\end{proof}

\begin{lemma}\label{lem::conv_envelop}
Recall that we denote by $\bs{\eta}_{\eps}$ be the envelop in Proposition~\ref{prop::BES_truncated}. 
The law of $\bs{\eta}_{\eps}$ converges as $\eps\to 0$ and the limit has the same law as $\PP_{\bs{\gamma}}$ in Theorem~\ref{thm::multiplelog_polygon}. 
Actually the convergence holds in total variation.
\end{lemma}
\begin{proof}
Denote by $\hat\Gamma$ the $\GFF$ with logarithmic singularities defined in~\eqref{eqn::multiplelog_GFF_polygon}. For every $0<\eps<\delta$, the law of $\hat\Gamma|_{\Omega_\delta}$ is absolutely continuous with respect to $\Gamma_\eps|_{\Omega_\delta}$ and the Radon-Nikodym derivative can be described as follows. By the domain Markov property of $\GFF$, we may construct a coupling $\QQ$ of $\hat\Gamma$ and $\Gamma_\eps$, such that under $\QQ$ we have
\begin{equation}\label{eqn::dec}
\hat\Gamma|_{\Omega_\eps}=\Gamma_\eps|_{\Omega_\eps}+h_\eps,
\end{equation}
where $h_\eps$ is harmonic on $\Omega_\eps$, and $h_\eps$ and $\Gamma_\eps|_{\Omega_\eps}$ are independent. Choose a nonnegative smooth function $\psi_\delta$ on $\Omega$ such that $\psi_\delta|_{\cup_{1\le\ell\le N}B(u_\ell,\delta/2)}=0$ and  $\psi_\delta|_{\Omega_\delta}=1$. Then,  we have (for instance, see~\cite[Proposition 1.51]{BerestyckiPowellGFFandLQG})
\[R_{\eps,\delta}:=\frac{\ud \hat\Gamma|_{\Omega_\delta}}{\ud \Gamma_\eps|_{\Omega_\delta}}=\exp\left((\Gamma_\eps,\psi_\delta h_\eps)_{\nabla}-\frac{1}{2}(\psi_\delta h_\eps,\psi_\delta h_\eps)_{\nabla}\right).\]
We will prove that
\begin{equation}\label{eqn::RN_boundconv}
\lim_{\eps\to 0}\QQ[|R_{\eps,\delta}-1|]=0.
\end{equation} 

Assume~\eqref{eqn::RN_boundconv} is true. By~\cite[Lemma 15]{AruSepulvedaWernerBTLStwoDGFF}, the (generalized) level lines $\bs {\eta}_{\eps}$ (resp. $\bs {\eta}$) is measurable with respect to $\Gamma_\eps$ (resp. $\Gamma$). Thus, suppose $F$ is a bounded and measurable function on curve space, we have
\begin{equation}
\QQ[F(\bs {\eta}_{\eps})\one_{\{\bs {\eta}_{\eps}\subset\Omega_\delta\}}]=\QQ[F(\bs {\eta})R_{\eps,\delta}\one_{\{\bs {\eta}\subset\Omega_\delta\}}].
\end{equation}
Taking $F\equiv 1$, by~\eqref{eqn::RN_boundconv}, we have
\[
\lim_{\eps\to 0}\QQ[\bs {\eta}_{\eps}\subset\Omega_\delta]=\QQ[\bs {\eta}\subset\Omega_\delta].
\]
Combining with Theorem~\ref{thm::multiplelog_polygon}, we have
\begin{equation}\label{eqn::nonhit1}
\lim_{\delta\to 0}\QQ[\bs {\eta}\not\subset\Omega_\delta]=0.
\end{equation}
Combining~\eqref{eqn::RN_boundconv}--\eqref{eqn::nonhit1} together, by first letting $\eps\to 0$ and then letting $\delta\to 0$, we have
\[\lim_{\eps\to 0}\QQ[F(\bs {\eta}_{\eps})]=\QQ[F(\bs {\eta})].\]
This completes the proof.

It remains to prove~\eqref{eqn::RN_boundconv}. Given $h_\eps$, the conditional law of $(\Gamma_\eps,\psi_\delta h_\eps)_{\nabla}$ is  $\LN(0,(\psi_\delta h_\eps,\psi_\delta h_\eps)_{\nabla})$. Thus, we have
\[\QQ[|R_{\eps,\delta}-1|^2]=\QQ[\exp\left((\psi_\delta h_\eps,\psi_\delta h_\eps)_{\nabla}\right)-1].\]
It suffices to prove that 
\begin{equation}\label{eqn::RNL2}
\lim_{\eps\to 0}\QQ[\exp\left((\psi_\delta h_\eps,\psi_\delta h_\eps)_{\nabla}\right)]=1.
\end{equation}
Since $h_\eps\in H_0^1(\Omega)$ is harmonic, we have
\[h_\eps(z)=\int_{\cup_{1\le\ell\le N}\partial B(u_\ell,\delta/4)}\Poisson(\Omega_{\delta/4};z,w)h_\eps(w)|\ud w|,\quad \text{for }z\in\Omega_{\delta/2}.\] 
Thus, there exists a constant $C>0$ which depends on $\Omega$ and $\delta$, such that 
\begin{equation}\label{eqn::compar}
0\le (\psi_\delta h_\eps,\psi_\delta h_\eps)_{\nabla}\le C\int_{\cup_{1\le\ell\le N}\partial B(u_\ell,\delta/4)} h^2_\eps(w)|\ud w|.
\end{equation}
For $w,s\in \cup_{1\le\ell\le N}\partial B(u_\ell,\delta/4)$, define the covariance function by
\begin{equation}\label{eqn::covariance}
K_\eps(w,s):=\Green(\Omega;w,s)-\Green(\Omega_\eps;w,s).
\end{equation}
Note that $K_\eps$ is a continuous function. In particular, it is a compact operator on $L^2=L^2(\cup_{1\le\ell\le N}\partial B(u_\ell,\delta/4))$. For $k\ge 1$, suppose $\phi^\eps_k\in L^2$ with $||\phi^\eps_k||_{L^2}=1$ and $\lambda^\eps_k\ge 0$ satisfy
\[\int_{\cup_{1\le\ell\le N}\partial B(u_\ell,\delta/4)}K_\eps(\cdot,w)\phi^\eps_k(w)|\ud w|=\lambda^\eps_k\phi^\eps_k(\cdot).\]
By definition, we have
\[\sum_{k\ge 1}(\lambda_k^\eps)^2=\int_{\cup_{1\le\ell\le N}\partial B(u_\ell,\delta/4)}\int_{\cup_{1\le\ell\le N}\partial B(u_\ell,\delta/4)}K(w,s)|\ud w||\ud s|.\]
Combining with~\eqref{eqn::covariance}, we have
\begin{equation}\label{eqn::tracezero}
\lim_{\eps\to 0}\sum_{k\ge 1}(\lambda_k^\eps)^2=0.
\end{equation}
Since $h_\eps$ is continuous on $\cup_{1\le\ell\le N}\partial B(u_\ell,\delta/4)$, we have the following decomposition
\[h_\eps=\sum_{k\ge 1}\lambda^\eps_k Z^\eps_k\phi^\eps_k.\]
By definition, we have that $\{Z^\eps_k\}_{k\ge 1}$ is a family of independent $\LN(0,1)$ variables. This implies that
\[\int_{\cup_{1\le\ell\le N}\partial B(u_\ell,\delta/4)} h^2_\eps(w)|\ud w|=\sum_{k\ge 1}(\lambda^\eps_k)^2(Z^\eps_k)^2.\]
Thus, we have
\begin{equation}\label{eqn::explicitL2}
\QQ\left[\exp\left(C\int_{\cup_{1\le\ell\le N}\partial B(u_\ell,\delta/4)} h^2_\eps(w)|\ud w|\right)\right]=\prod_{k\ge 1}\left(1-2C(\lambda_k^\eps)^2\right)^{1/2}.
\end{equation}
Combining~\eqref{eqn::compar},~\eqref{eqn::tracezero} and~\eqref{eqn::explicitL2} together, we obtain~\eqref{eqn::RNL2}. This completes the proof.
\end{proof}
Now, we come back to the proof of Theorem~\ref{thm::Brownian_represent}.
\begin{proof}[Proof of Theorem~\ref{thm::Brownian_represent}]
This simply follows by combining Lemma~\ref{lem::conv_envelop_1}, Lemma~\ref{lem::conv_envelop}, 
and Proposition \ref{prop::BES_truncated}.
\end{proof}

\begin{remark}
The Brownian representation show that when all the $\gamma_\ell$-s are non-negative, 
the subdomain of $\Omega$ delimited by
$\partial^{+}\Omega$ and $\bs {\eta}$ is non-decreasing by inclusion
with respect to each of the $\gamma_\ell$-s: there are more Brownian trajectories when
$\gamma_\ell$ increases.
This monotonicity actually also holds without the condition of $\gamma_\ell$-s being non-negative.
This follows from the monotonicity of the first passage sets;
see \cite[Proposition 4.5]{AruLupuSepulvedaFPSGFF}.
\end{remark}
\section{Asymptotic}
\label{sec::asymp}
In this section, we prove Proposition~\ref{prop::asymp}. 
We will prove the conclusion in the unit disc and then transform it to general polygon via conformal covariance. 
The conclusion in the unit disc is the following.

\begin{proposition}\label{prop::asymp_forall_U}
Fix $N\ge 1$ and $\bs{\theta}=(\theta_1, \ldots, \theta_{2N})\in\LX_{2N}$. We write $x_j=\ee^{\ii\theta_j}$ for $1\le j\le 2N$. 
Suppose $\Gamma$ is GFF in $(\U; \bs{x})$ with alternating boundary data~\eqref{eqn::multiplelog_polygon_boundarydata}. Fix $\gamma\in\R$ and consider the level lines of the field
$\Gamma+\gamma\log|\cdot|$. 
Let $\eta^j$ be the level line starting from $x_{2j-1}$ for $1\le j\le N$. Then 
\begin{align}\label{eqn::asymp_forall_U}
\PP_{\gamma}\left[\forall j\in \{1, \ldots, N\}, \dist(\eta^j, 0)\le \eps\right]\sim \eps^{\frac{N^2+\gamma^2}{2}}\LZradial(\bs{\theta})/\LZ_{\GFF}^{(\gamma)}(\bs{\theta}); 
\end{align} 
where the constant in $\sim$ depends on $(N, \gamma)$, and  
\begin{align}\label{eqn::multiradial_kappa4_pf}
\LZradial(\bs{\theta})=\prod_{1\le i<j\le 2N}\left(\sin\left(\frac{\theta_j-\theta_i}{2}\right)\right)^{\frac{1}{2}}; 
\end{align}
and
\begin{align}\label{eqn::GFFonelog_pf}
\LZ_{\GFF}^{(\gamma)}(\bs{\theta})=\prod_{1\le i<j\le 2N}\left(\sin\left(\frac{\theta_j-\theta_i}{2}\right)\right)^{\frac{1}{2}(-1)^{j-i}}\exp\left(\gamma\sum_{j=1}^N(\theta_{2j}-\theta_{2j-1})\right). 
\end{align}
\end{proposition}

The main ingredients in the proof of Proposition~\ref{prop::asymp_forall_U} are tools  developed recently from~\cite{HealeyLawlerNSidedRadialSLE, HuangPeltolaWuMultiradialSLE}. 
We introduce the first key tool ``multi-time martingale" from~\cite{HuangPeltolaWuMultiradialSLE} in Section~\ref{subsec::multitime_mart} and then recall the second key tool ``common-time parameter" from~\cite{HealeyLawlerNSidedRadialSLE} in Section~\ref{subsec::DysonBM}. We complete the proof of Proposition~\ref{prop::asymp_forall_U} in Section~\ref{subsec::asymp}. 

\subsection{Multi-time martingale}
\label{subsec::multitime_mart}
\paragraph*{Radial Loewner chain.}
For a simply connected domain $\Omega\subsetneq\C$ and $z\in\Omega$, the conformal radius of $\Omega$ seen from $z$ is defined by $\CR(\Omega; z):=1/\varphi'(z)$ where $\varphi:\Omega\to \U$ is the conformal map such that $\varphi(z)=0$ and $\varphi'(z)>0$. In particular, $\CR(\U; 0)=1$. 
A $\U$-hull is a relatively closed subset $K$ of $\U$ such that $\U\setminus K$ is simply connected and containing $0$. 
Let $\mathfrak{g}_K$ be
the conformal map from $\U\setminus K$ onto $\U$ satisfying the normalization $\mathfrak{g}_K(0)=0$ and $\mathfrak{g}'_K(0)=\CR(\U\setminus K; 0)^{-1}>0$. We refer to $\mathfrak{g}_K$ as the mapping-out function of $K$. 

Fix $\theta\in \R$ and $T \in (0,\infty]$. Let $\xi: [0,T) \to \R$ be a continuous function with $\xi_0 = \theta$.
A radial Loewner chain driven by $\xi$ is a family of $\U$-hulls $(K_t)_{t\in [0,T)}$ such that the mapping-out function $\mathfrak{g}_t:=\mathfrak{g}_{K_t}$ satisfies the radial Loewner equation
\begin{align*}
	\partial_t \mathfrak{g}_t(z)=\mathfrak{g}_t(z) \frac{\ee^{\ii \xi_t}+\mathfrak{g}_t(z)}{\ee^{\ii \xi_t}-\mathfrak{g}_t(z)}, \qquad \mathfrak{g}_0(z)=z, \qquad z\in \overline{\U}.
\end{align*}
It is convenient to use the covering map
$\covmap_t \colon \R \to \R$ 
of $\mathfrak{g}_t$ defined via $\mathfrak{g}_t(\ee^{\ii \theta}) = \exp( \ii \covmap_t(\theta) )$. The covering map satisfies the radial Loewner equation
\begin{align*}\label{eqn::single_radial_Loewner_equation_cov}
	\partial_t \covmap_t(\theta) = \cot \bigg(\frac{\covmap_t(\theta) - \xi_t}{2}\bigg) , \qquad \covmap_0(\theta)=\theta.
\end{align*}

%
Fix $\kappa>0$ and $\theta\in\R$, radial $\SLE_{\kappa}$ in $(\U; \ee^{\ii\theta}; 0)$ is the Loewner chain driven by $\xi_t=\theta+\sqrt{\kappa}B_t$ where $B_t$ is standard one-dimensional Brownian motion. When $\kappa\le 4$, it is almost surely generated by a continuous simple curve $\eta$ and it is transient: $\lim_{t\to\infty}\eta_t=0$, see~\cite[Theorem~1.3]{LawlerContinuityofradialSLE}. 

\paragraph*{Multi-time parameter.}
Fix $n\ge 1$ and $\bs{\theta}=(\theta_1,\ldots,\theta_n)\in \LX_n$. 
Consider an $n$-tuple \[\bs{\eta}_{[\bs{0},\bs{t}]}=(\eta^1_{[0,t_1]}, \ldots, \eta^n_{[0,t_n]})\] of simple curves in the polygon $(\U;\ee^{\ii\bs{\theta}};0)$ parameterized by $\bs{t}=(t_1, \ldots, t_n)\in[0,\infty)^n$, such that $\eta_0^j=\ee^{\ii \theta_j}$ and the segments $\{ \eta^{j}_{[0,t_j]}\}_{1\le j\le n}$ are disjoint. 
We define the following normalized conformal transformations:
\begin{itemize}
	\item $\mathfrak{g}_{t_j}^j$ is the conformal map from $\U\setminus\eta^j_{[0,t_j]}$ onto $\U$ with $\mathfrak{g}_{t_j}^j(0)=0$ and $(\mathfrak{g}_{t_j}^j)'(0)>0$, where $1\le j\le n$.
	\item $\mathfrak{g}_{\bs{t}}$ is the conformal map from $\U\setminus\cup_{j=1}^n \eta^j_{[0,t_j]}$ onto $\U$ with $\mathfrak{g}_{\bs{t}}(0)=0$ and $\mathfrak{g}_{\bs{t}}'(0)>0$.
	\item $\mathfrak{g}_{\bs{t}, j}$ is the conformal map from $\U \setminus \mathfrak{g}_{t_j}^j\left(\cup_{i\neq j}\eta^i_{[0,t_i]}\right)$ onto $\U$ with $\mathfrak{g}_{\bs{t}, j}(0)=0$ and $\mathfrak{g}'_{\bs{t}, j}(0)> 0$, where $1\le j\le n$. 
\end{itemize}
Using these notations, we have $\mathfrak{g}_{\bs{t}}=\mathfrak{g}_{\bs{t}, j}\circ \mathfrak{g}_{t_j}^j$ for $1\le j\le n$. 
Note that multi-time parameter is well-defined up to the collision time which is the first time when any two curves in $\bs{\eta}$ meet. 
Let $\covmap_{t_j}^{j},\covmap_{\bs{t}},\covmap_{\bs{t},j}$ be the covering maps of $\mathfrak{g}_{t_j}^{j},\mathfrak{g}_{\bs{t}},\mathfrak{g}_{\bs{t},j}$, respectively. 
We say that the $n$-tuple $\bs{\eta}_{\bs{t}}$ of curves has $n$-time parameter if $g^{j}$ is parameterized by capacity:
\begin{equation*}
	(\mathfrak{g}_{t_j}^j)'(0)= \ee^{t_j},\quad \text{ for }1\leq j \leq n.
\end{equation*}
Denote by $\xi^j$ the driving function of each $\eta^j$, and we define the driving function of the $n$-tuple $\bs{\eta}_{\bs{t}}$, started at $\bs{\mslitdriv}_0=(\theta_1, \ldots, \theta_n)\in\LX_n$, by
\begin{equation*}
	\bs{\mslitdriv}_{\bs{t}}=(\bs{\mslitdriv}^1_{\bs{t}}, \ldots, \bs{\mslitdriv}^n_{\bs{t}}), \qquad \text{with }\bs{\mslitdriv}^j_{\bs{t}}=\covmap_{\bs{t}, j}(\xi^j_{t_j}), \qquad \text{for }1\le j\le n.
\end{equation*}
The mapping-out function $\mathfrak{g}_{\bs{t}}$ solves the Loewner equation
\begin{equation*}
	\ud \mathfrak{g}_{\bs{t}}(z)=\sum_{j=1}^n \partial_{t_j}\mathfrak{g}_{\bs{t}}(z)\ud t_j=\mathfrak{g}_{\bs{t}}(z)\sum_{j=1}^n(\covmap'_{\bs{t},j}(\xi_{t_j}^j))^2\left(\frac{\ee^{\ii\mslitdriv_{\bs{t}}^{j}}+\mathfrak{g}_{\bs{t}}(z)}{\ee^{\ii\mslitdriv_{\bs{t}}^{j}}-\mathfrak{g}_{\bs{t}}(z)} \right)\ud t_j. 
\end{equation*}

\paragraph*{Multi-time local martingale.}
Fix $\kappa>0$ and we use the following CFT/SLE parameters indexed by $\kappa > 0$:
\begin{align}\label{eqn::universal_parameters}
	\mathfrak{b} := \frac{6-\kappa}{2\kappa} ,\qquad 
	\tilde{\mathfrak{b}} := \frac{(6-\kappa)(\kappa-2)}{8\kappa} , 
	\qquad \textnormal{and} \qquad 
	\mathfrak{c} := \frac{(6-\kappa)(3\kappa-8)}{2\kappa}.
\end{align}
Fix $n\ge 2$ and $\bs{\theta}=(\theta_1, \ldots, \theta_n)\in\LX_n$. 
For each $j\in \{1,\ldots, n\}$, let $\eta^{j}$ be radial $\SLE_{\kappa}$ in $(\U; \ee^{\ii\theta_j}; 0)$ and let $\mathsf{P}_n$ be the probability measure on $\bs{\eta}=(\eta^{1}, \ldots, \eta^{n})$ under which the curves are independent. 
We parameterize $\bs{\eta}$ by $n$-time parameter $\bs{t}$, and let $\bs{\mslitdriv}_{\bs{t}} = (\mslitdriv_{\bs{t}}^{1},\ldots,\mslitdriv_{\bs{t}}^{n})$ be the multi-slit driving function.
For a function $\LZ \colon \LX_n\to \R_{>0}$ and a constant $\aleph\in\R$, define 
\begin{align}\label{eqn::multitime_mart_radial}
	\begin{split}
		M_{\bs{t}}(\LZ)
		:= \; & \one_{\LE_{\emptyset}(\bs{\eta}_{\bs{t}})} \, \exp\bigg(\frac{\mathfrak{c}}{2}\bm_{\bs{t}} -\tilde{\mathfrak{b}}\sum_{j=1}^{n}t_j\bigg) 
		\; (\mathfrak{g}_{\bs{t}}'(0))^{\tilde{\mathfrak{b}}-\aleph} 
		\Big(\prod_{j=1}^{n} \covmap_{\bs{t},j}'(\xi_{t_j}^{j}) \Big)^{\mathfrak{b}}
		\LZ(\bs{\mslitdriv}_{\bs{t}})
	\end{split}
\end{align}
where $\LE_{\emptyset}(\bs{\eta}_{\bs{t}}) = \{ \eta_{[0,t_j]}^{j} \cap \eta_{[0,t_i]}^{i}=\emptyset, \, \forall i\neq j\}$ is the event that different curves are disjoint, and 
$\bm_{\bs{t}}$ is the unique potential solving the exact differential equation \textnormal{(}see~\cite[Lemma~2.6]{HuangPeltolaWuMultiradialSLE}\textnormal{):} 
\begin{align*}
	\ud  \bm_{\bs{t}} 
	= \sum_{j=1}^{n} \partial_{t_j} \bm_{(t_1,\ldots,t_p)} \ud t_j
	= \sum_{j=1}^{n} \Big( \! -\frac{1}{3} (\LS \covmap_{\bs{t},j})(\xi_{t_j}^{j})+\frac{1}{6} \big(1- (\covmap'_{\bs{t},j}(\xi_{t_j}^{j}))^2 \big) \Big) \ud t_j, 
	\qquad \bm_{\bs{0}} = 0,
\end{align*}
with $\LS\covmap=\frac{\covmap'''}{\covmap'}-\frac{3}{2}\big(\frac{\covmap''}{\covmap'}\big)^2$ denoting the Schwarzian derivative of a function $\covmap$.
The proof of Proposition~\ref{prop::asymp_forall_U} relies on two multi-time local martingales that we summarize in Proposition~\ref{prop::multitime_mart_two}
\begin{proposition}\label{prop::multitime_mart_two}
Fix $N\ge 1$ and $\bs{\theta}=(\theta_1, \ldots, \theta_{2N})\in\LX_{2N}$. 
For each $j\in \{1,\ldots, 2N\}$, let $\eta^{j}$ be radial $\SLE_{\kappa}$ in $(\U; \ee^{\ii\theta_j}; 0)$ and let $\mathsf{P}_{2N}$ be the probability measure on $\bs{\eta}=(\eta^{1}, \ldots, \eta^{2N})$ under which the curves are independent. 
We parameterize $\bs{\eta}$ by $2N$-time parameter $\bs{t}$, and let $\bs{\mslitdriv}_{\bs{t}} = (\mslitdriv_{\bs{t}}^{1},\ldots,\mslitdriv_{\bs{t}}^{2N})$ be the multi-slit driving function.
\begin{itemize}
\item[(1)]
The following process is $2N$-time local martingale under $\mathsf{P}_{2N}$ up to any collision time: 
\begin{align}\label{eqn::multiradial_kappa4_mart}
M_{\bs{t}}(\LZradial)=\one_{\LE_{\emptyset}(\bs{\eta}_{\bs{t}})} \, \exp\bigg(\frac{1}{2}\bm_{\bs{t}} -\frac{1}{8}\sum_{j=1}^{2N}t_j\bigg) 
		\; (\mathfrak{g}_{\bs{t}}'(0))^{\frac{N^2}{2}} 
		\Big(\prod_{j=1}^{2N} \covmap_{\bs{t},j}'(\xi_{t_j}^{j}) \Big)^{\frac{1}{4}}
		\LZradial(\bs{\mslitdriv}_{\bs{t}}). 
\end{align}
Denote by $\mathsf{P}_{2N}(\LZradial)$ the law obtained by tilting $\mathsf{P}_{2N}$ by $M_{\bs{t}}(\LZradial)$. 
\item[(2)]The following process is $2N$-time local martingale under $\mathsf{P}_{2N}$ up to any collision time: 
\begin{align}\label{eqn::multitime_mart_kappa4_gamma}
M_{\bs{t}}(\LZ_{\GFF}^{(\gamma)})=\one_{\LE_{\emptyset}(\bs{\eta}_{\bs{t}})} \, \exp\bigg(\frac{1}{2}\bm_{\bs{t}} -\frac{1}{8}\sum_{j=1}^{2N}t_j\bigg) 
		\; (\mathfrak{g}_{\bs{t}}'(0))^{-\frac{\gamma^2}{2}} 
		\Big(\prod_{j=1}^{2N} \covmap_{\bs{t},j}'(\xi_{t_j}^{j}) \Big)^{\frac{1}{4}}
		\LZ_{\GFF}^{(\gamma)}(\bs{\mslitdriv}_{\bs{t}}). 
\end{align}
Denote by $\mathsf{P}_{2N}(\LZ_{\GFF}^{(\gamma)})$ the law obtained by tilting $\mathsf{P}_{2N}$ by $M_{\bs{t}}(\LZ_{\GFF}^{(\gamma)})$. 
\item[(3)]
Consider the level lines $\bs{\eta}$ in Proposition~\ref{prop::asymp_forall_U}.
We view it as a $2N$-tuple of curves: let $\eta^j$ be the level line of $\Gamma+\gamma\log|\cdot|$ starting from $x_{j}$ when $j$ is odd and let $\eta^j$ be the level line of $-\Gamma-\gamma\log|\cdot|$ starting from $x_j$ when $j$ is even. Then the law of the $2N$-tuple $\bs{\eta}$ under $\PP_{\gamma}$
is the same as $\mathsf{P}_{2N}(\LZ_{\GFF}^{(\gamma)})$ up to any collision time. 
\end{itemize}
\end{proposition}

\begin{lemma}[{\cite[Proposition~2.4]{HuangPeltolaWuMultiradialSLE}}]\label{lem::multitime_mart_radial}
Fix $\kappa>0$ and $\aleph\in\R$. 
The  process $M(\LZ)$ defined in~\eqref{eqn::multitime_mart_radial} is $n$-time local martingale under $\mathsf{P}_n$ if $\LZ$ is positive and smooth and it satisfies the following system of \emph{radial BPZ equations}\textnormal{:}
for each $\bs{\vartheta} := (\vartheta_1, \ldots, \vartheta_n)\in\LX_n$, 
\begin{align}\label{eqn::radialBPZ}
	\frac{\kappa}{2} \frac{\partial_j^2 \LZ(\bs{\vartheta})}{\LZ(\bs{\vartheta})} 
	+ \underset{1\leq i\neq j \leq n}{\sum} \, \bigg( \cot \bigg(\frac{\vartheta_{i}-\vartheta_{j}}{2}\bigg)
	\frac{\partial_{i} \LZ(\bs{\vartheta})}{\LZ(\bs{\vartheta})} \, - \, \frac{\mathfrak{b}/2}{ \sin^2 \big(\frac{\vartheta_{i}-\vartheta_{j}}{2}\big)}\bigg)=\aleph, \qquad \textnormal{for all }j\in \{1,\ldots,n\}.
\end{align}
\end{lemma}

\begin{lemma}\label{lem::tworadailBPZ}
\begin{itemize}
\item[(1)] The partition function $\LZradial$~\eqref{eqn::multiradial_kappa4_pf} satisfies the system of radial BPZ equations~\eqref{eqn::radialBPZ} with $n=2N, \kappa=4$ and $\aleph=\frac{1}{8}-\frac{N^2}{2}$.
\item[(2)] The partition function $\LZ_{\GFF}^{(\gamma)}$~\eqref{eqn::GFFonelog_pf} satisfies the system of radial BPZ equations~\eqref{eqn::radialBPZ} with $n=2N, \kappa=4$ and $\aleph=\frac{1}{8}+\frac{\gamma^2}{2}$.
\end{itemize}
\end{lemma}
\begin{proof}
The BPZ equations for $\LZradial$ is straight forward and has been investigated in more generality in earlier literature~\cite{HuangPeltolaWuMultiradialSLE}. Let us check the BPZ equations for $\LZ_{\GFF}^{(\gamma)}$. We denote the differential operator by 
\begin{align}
\LD_j\LZ:=	2\frac{\partial_j^2 \LZ(\bs{\vartheta})}{\LZ(\bs{\vartheta})} 
	+ \underset{1\leq i\neq j \leq 2N}{\sum} \, \bigg( \cot \bigg(\frac{\vartheta_{i}-\vartheta_{j}}{2}\bigg)
	\frac{\partial_{i} \LZ(\bs{\vartheta})}{\LZ(\bs{\vartheta})} \, - \, \frac{1/8}{ \sin^2 \big(\frac{\vartheta_{i}-\vartheta_{j}}{2}\big)}\bigg). 
	\end{align}
	Due to symmetry, it suffices to show $\LD_1\LZ_{\GFF}^{(\gamma)}=\frac{1}{8}+\frac{\gamma^2}{2}$. We have the following two observations.
	\begin{itemize}
	\item On the one hand, Lemma~\ref{lem::multiple_log_mart} with $M=1$ and a standard change of coordinate calculation~\cite{SchrammWilsonSLECoordinatechanges} tells that the following process is a local martingale for radial $\SLE_4$ in $(\U; \ee^{\ii\theta_1}; 0)$: 
	\begin{align}
	M_t(\LZ_{\GFF}^{(\gamma)})=\prod_{j=2}^{2N}\covmap'_t(\theta_j)^{\frac{1}{4}}\times \mathfrak{g}'_t(0)^{-\frac{1}{8}-\frac{\gamma^2}{2}}\times \LZ_{\GFF}^{(\gamma)}(\xi_t, \covmap_t(\theta_2), \ldots, \covmap_t(\theta_{2N})). 
	\end{align}
	\item On the other hand, It\^o's formula tells 
	\begin{align*}
	\frac{\ud M_t(\LZ_{\GFF}^{(\gamma)})}{M_t(\LZ_{\GFF}^{(\gamma)})}=&2\partial_1(\log \LZ_{\GFF}^{(\gamma)})(\xi_t, \covmap_t(\theta_2), \ldots, \covmap_t(\theta_{2N}))\ud B_t\\
	&+\left(\LD_1\LZ_{\GFF}^{(\gamma)}(\xi_t, \covmap_t(\theta_2), \ldots, \covmap_t(\theta_{2N}))-\frac{1}{8}-\frac{\gamma^2}{2}\right)\ud t. 
\end{align*}	 
	\end{itemize}
As $M_t(\LZ_{\GFF}^{(\gamma)})$ is local martingale, the drift term of $\ud M_t(\LZ_{\GFF}^{(\gamma)})$ must vanish. Thus, $\LD_1\LZ_{\GFF}^{(\gamma)}=\frac{1}{8}+\frac{\gamma^2}{2}$ as desired. 
\end{proof}

\begin{proof}[Proof of Proposition~\ref{prop::multitime_mart_two}]
The processes $M_{\bs{t}}(\LZradial)$ and $M_{\bs{t}}(\LZ_{\GFF}^{(\gamma)})$ are $2N$-time local martingales under $\mathsf{P}_{2N}$ due to Lemmas~\ref{lem::multitime_mart_radial} and~\ref{lem::tworadailBPZ}. It remains to show that the law of $\bs{\eta}$ under $\PP_{\gamma}$ is the same as $\mathsf{P}_{2N}(\LZ_{\GFF}^{(\gamma)})$. 

Note that 
\begin{align}
\label{eqn::multitimemart_aux2}
\begin{split}
	\ud \mslitdriv_{\bs{t}}^{j} = & \covmap_{\bs{t},j}'\left(\xi_{t_j}^{j}\right) \ud \xi_{t_j}^{j} - \covmap_{\bs{t},j}''\left(\xi_{t_j}^{j}\right) \ud t_j  +\sum_{l \neq j} \cot\left(\left(\mslitdriv_{\bs{t}}^{j}-\mslitdriv_{\bs{t}}^{l}\right)/2\right)\covmap_{\bs{t},l}'\left(\xi_{t_l}^{l}\right)^2 \ud t_l,
\end{split}
\end{align}
and 
\begin{align}
\label{eqn::dMtLZgffr}
\frac{\ud M_{\bs{t}}(\LZ_{\GFF}^{(\gamma)})}{M_{\bs{t}}(\LZ_{\GFF}^{(\gamma)})}=& \sum_{j=1}^{2N} \left( \frac{\partial_j \LZ_{\GFF}^{(\gamma)}}{\LZ_{\GFF}^{(\gamma)}} \covmap_{\bs{t},j}'\left(\xi_{t_j}^{j}\right)+ \frac{1}{4} \frac{\covmap_{\bs{t},j}''\left(\xi_{t_j}^{j}\right)}{\covmap_{\bs{t},j}'\left(\xi_{t_j}^{j}\right)} \right) \ud \xi_{t_j}^{j}.
\end{align}
From Girsanov's theorem and \eqref{eqn::multitimemart_aux2} and \eqref{eqn::dMtLZgffr}, we have
\begin{align*}
\begin{split}
	\ud \mslitdriv_{\bs{t}}^{j}=&2 \covmap_{\bs{t},j}'\left(\xi_{t_j}^{j}\right) \ud \tilde{B}_{t_j}^{j} + 4 \left( \partial_j \log\LZ_{\GFF}^{(\gamma)} \right) \left( \mslitdriv_{\bs{t}}^{1},\cdots,\mslitdriv_{\bs{t}}^{2N} \right) \covmap_{\bs{t},j}'\left(\xi_{t_j}^{j}\right)^2 \ud t_j \\
	&+ \sum_{l \neq j} \cot\left(\left(\mslitdriv_{\bs{t}}^{j}-\mslitdriv_{\bs{t}}^{l}\right)/2\right) \covmap_{\bs{t},l}'\left(\xi_{t_l}^{l}\right)^2 \ud t_l,
\end{split}
\end{align*}
where $\tilde{B}^1, \ldots, \tilde{B}^{2N}$ are independent Brownian motions under $\mathsf{P}_{2N}(\LZ_{\GFF}^{(\gamma)})$. Taking into account the variation of the capacity parametrization, we obtain that $\PP_{\gamma}$ is the same as $\mathsf{P}_{2N}(\LZ_{\GFF}^{(\gamma)})$ as desired. 
\end{proof}

\subsection{Dyson Brownian motion}
\label{subsec::DysonBM}
\paragraph*{Common-time parameter.}
Fix $n\ge 1$ and $\bs{\theta}=(\theta_1,\ldots,\theta_n)\in \LX_n$. 
Consider an $n$-tuple \[\bs{\eta}_{[\bs{0},\bs{t}]}=(\eta^1_{[0,t_1]}, \ldots, \eta^n_{[0,t_n]})\] of simple curves in the polygon $(\U;\ee^{\ii\bs{\theta}};0)$ parameterized by $\bs{t}=(t_1, \ldots, t_n)\in[0,\infty)^n$, such that $\eta_0^j=\ee^{\ii \theta_j}$ and the segments $\{ \eta^{j}_{[0,t_j]}\}_{1\le j\le n}$ are disjoint. 
Recall that $\mathfrak{g}_{\bs{t}}$ is the conformal map from $\smash{\U \setminus \cup_{j=1}^n \eta^{j}_{[0,t_j]}}$ onto $\U$ normalized at the origin, i.e., 
\[
	\mathfrak{g}_{\bs{t}}(0)=0, \quad \log \mathfrak{g}_{\bs{t}}'(0)>0.
\] 
A standard calculation gives
\begin{equation*}
	\ud \log \mathfrak{g}_{\bs{t}}'(0)=\sum_{j=1}^{n} \covmap'_{\bs{t},j} ( \xi_{t_j}^{j} )^2 \ud t_j.
\end{equation*}
We say that the $n$-tuple $\bs{\eta}_{[\bs{0},\bs{t}]}$ is parameterized by common-time parameter if
\begin{equation*}
	\partial_{t_j} \log \mathfrak{g}_{(t,\ldots,t)}'(0)=1, \quad \text{for } 1\le j\le n.
\end{equation*}
It is proved in~\cite[Lemma~3.2]{HealeyLawlerNSidedRadialSLE} that common-time parameter exists for continuous disjoint simple curves $\bs{\eta}_{[\bs{0},\bs{t}]}$. 
We define the collision time to be the first time that any two curves in $\bs{\eta}$ meet. 
The relation between common-time parameter and multi-time parameter is given by
\begin{equation*}
	\ud t_j = \covmap_{\bs{t},j}'(W_{t_j}^j)^{-2} \ud t, \quad \text{ for } 1\le j\le n,
\end{equation*}
Moreover,
\begin{equation*}
	t\le t_j(t) \le nt, \quad \text{for } 1\le j\le n.
\end{equation*}
Under common-time parameter, we denote the $n$-tuple by $\bs{\eta}_{[0,t]}=\bs{\eta}_{[(0,\ldots,0),(t,\ldots,t)]}$. We define the following normalized conformal transformations:
\begin{itemize}
	\item $\mathfrak{g}_t^j$ is the conformal map from $\U\setminus \eta_{[0,t]}^j$ onto $\U$ with $\mathfrak{g}_t^j(0)=0$ and $\left(\mathfrak{g}_t^j \right)'(0)>0$, where $1\le j\le n$.
	\item $\mathfrak{g}_t$ is the conformal map from $\U\setminus \cup_{j=1}^n \eta_{[0,t]}^j$ onto $\U$ with $\mathfrak{g}_t(0)=0$ and $\mathfrak{g}'_t(0)=\ee^{nt}>0$. 
	\item $\mathfrak{g}_{t,j}$ is the conformal map from $\U\setminus \mathfrak{g}_{t}^j (\cup_{i \neq j} \eta_{[0,t]}^i)$ onto $\U$ with $\mathfrak{g}_{t,j}(0)=0$ and $\mathfrak{g}'_{t,j}(0)>0$, where $1\le j\le n$.
\end{itemize}
Let $\covmap_{t}^{j},\covmap_{t},\covmap_{t,j}$ be the covering maps of $\mathfrak{g}_{t}^{j},\mathfrak{g}_{t},\mathfrak{g}_{t,j}$, respectively. Denote by $\xi^j$ the driving function of $\eta^j$ and denote the driving function of the $n$-tuple $\bs{\eta}_{[0,t]}$, started from $\bs{\mslitdriv}_0=(\theta_1,\ldots,\theta_n)\in\LX_n$, by
\begin{equation*}
	\bs{\mslitdriv}_t=(\mslitdriv_t^1,\ldots,\mslitdriv_t^n), \quad \text{with } \mslitdriv_t^j=\covmap_{t,j}\left(\xi_t^j\right), \quad\text{for } 1\le j\le n.
\end{equation*}

\begin{lemma}\label{lem::HL21}
Suppose $\bs{\eta}=(\eta^1, \ldots, \eta^{2N})$ has the law of $\mathsf{P}_{2N}(\LZradial)$ in Proposition~\ref{prop::multitime_mart_two}. 
We parameterize $\bs{\eta}$ using common-time parameter and denote its driving function by $\bs{\mslitdriv}_t=(\mslitdriv_t^1,\ldots,\mslitdriv_t^{2N})$. Then 
\begin{equation*}
	\ud \mslitdriv_t^j = 2 \, \ud B_t^j + 2 \sum_{i \neq j} \cot \left( \frac{\mslitdriv_t^j-\mslitdriv_t^i}{2} \right)  \, \ud t, \quad 1 \leq j \leq 2N. 
\end{equation*}
Denote by $\Pradial(t; \bs{\theta}, \bs{\theta'})$ the transition density for $\bs{\mslitdriv}_t$ under $\mathsf{P}_{2N}(\LZradial)$. We define 
\begin{align}\label{eqn::Dyson_invdensity}
\mathsf{p}(\bs{\theta})=\frac{1}{\mathsf{C}_N}\prod_{1\le i<j\le 2N}\left(\sin\left(\frac{\theta_j-\theta_i}{2}\right)\right)^{2}, \qquad\text{where }\mathsf{C}_N=\int_{\LX_{2N}}\prod_{1\le i<j\le 2N}\left(\sin\left(\frac{\theta_j-\theta_i}{2}\right)\right)^{2}\ud \bs{\theta}. 
\end{align}
Then we have the following asymptotic for $\Pradial(t; \bs{\theta}, \bs{\theta'})$ as $t\to\infty$: 
\begin{align}\label{eqn::HL21_asy}
\Pradial(t; \bs{\theta}, \bs{\theta'})=\mathsf{p}(\bs{\theta'})\left(1+O(\ee^{-at})\right), 
\end{align}
where $a>0$ is a constant depending on $N$. 
\end{lemma}
\begin{proof}
This is a special case of~\cite[Proposition~5.5]{HealeyLawlerNSidedRadialSLE}. 
\end{proof}

\begin{lemma}\label{lem::transition_density_gff}
Assume the same setup as in Proposition~\ref{prop::multitime_mart_two}. Denote $\bs{\eta}=(\eta^1, \ldots, \eta^{2N})$ and parameterize it by common-time parameter. Denote by $\bs{\mslitdriv}_t=(\mslitdriv_t^1,\ldots,\mslitdriv_t^{2N})$ its driving function and by $T$ the collision time. Denote by $\mathsf{P}_{\GFF}(t; \bs{\theta}, \bs{\theta'})$ the transition density for the driving function $\bs{\mslitdriv}_t=(\mslitdriv_t^1,\ldots,\mslitdriv_t^{2N})$ in the sense that, for any test function $F$, 
\begin{align*}
\E_{\gamma}\left[F(\bs{\mslitdriv}_t)\one\{T>t\}\right]=\int_{\LX_{2N}}\mathsf{P}_{\GFF}(t; \bs{\theta}, \bs{\theta'})F(\bs{\theta'})\ud\bs{\theta'}.
\end{align*}
Then
\begin{align}\label{eqn::transition_density_gff}
\mathsf{P}_{\GFF}(t; \bs{\theta}, \bs{\theta'})=\Pradial(t; \bs{\theta}, \bs{\theta'})\frac{\LZradial(\bs{\theta})/\LZradial(\bs{\theta'})}{\LZ_{\GFF}^{(\gamma)}(\bs{\theta})/\LZ_{\GFF}^{(\gamma)}(\bs{\theta'})}\exp\left(-(N^2+\gamma^2)Nt\right), 
\end{align}
where $\Pradial(t; \bs{\theta}, \bs{\theta'})$ is the transition density in Lemma~\ref{lem::HL21} and $\LZradial$ and $\LZ_{\GFF}^{(\gamma)}$ are partition functions in~\eqref{eqn::multiradial_kappa4_pf} and~\eqref{eqn::GFFonelog_pf}. Consequently, we have the following asymptotic for $\mathsf{P}_{\GFF}(t; \bs{\theta}, \bs{\theta'})$ as $t\to\infty$:
\begin{align}\label{eqn::transition_density_gff_asy}
\mathsf{P}_{\GFF}(t; \bs{\theta}, \bs{\theta'})=\mathsf{p}(\bs{\theta'})\frac{\LZradial(\bs{\theta})/\LZradial(\bs{\theta'})}{\LZ_{\GFF}^{(\gamma)}(\bs{\theta})/\LZ_{\GFF}^{(\gamma)}(\bs{\theta'})}\exp\left(-(N^2+\gamma^2)Nt\right)\left(1+O(\ee^{-at})\right), 
\end{align}
where $\mathsf{p}(\bs{\theta'})$ is the invariant density defined in~\eqref{eqn::Dyson_invdensity} and $\LZradial, \LZ_{\GFF}$ are partition functions defined in~\eqref{eqn::multiradial_kappa4_pf} and~\eqref{eqn::GFFonelog_pf}. 
\end{lemma}

\begin{proof}
When using mulit-time parameter, the law of $\bs{\mslitdriv}_{\bs{t}}$ under $\PP_{\gamma}$ is the same as $\mathsf{P}_{2N}(\LZradial)$ weighted by the $2N$-time local martingale:
\begin{align*}
N_{\bs{t}}=\frac{M_{\bs{t}}(\LZ_{\GFF}^{(\gamma)})}{M_{\bs{t}}(\LZradial)}=(\mathfrak{g}_{\bs{t}}'(0))^{-\frac{N^2}{2}-\frac{\gamma^2}{2}} 
		\frac{\LZ_{\GFF}^{(\gamma)}(\bs{\mslitdriv}_{\bs{t}})}{\LZradial(\bs{\mslitdriv}_{\bs{t}})}. 
\end{align*}
Using common-time parameter, it becomes
\begin{align*}
N_t=\exp\left(-(N^2+\gamma^2)Nt\right)\frac{\LZ_{\GFF}^{(\gamma)}(\bs{\mslitdriv}_t)}{\LZradial(\bs{\mslitdriv}_t)}. 
\end{align*}
Thus, the law of $\bs{\eta}_{[0,t]}$ under $\PP_{\gamma}$ is the same as the law of $\mathsf{P}_{2N}(\LZradial)$ weighted by the Radon-Nikodym derivative
\begin{align*}
\frac{N_t}{N_0}=\exp\left(-(N^2+\gamma^2)Nt\right)\frac{\LZradial(\bs{\theta})/\LZradial(\bs{\mslitdriv}_t)}{\LZ_{\GFF}^{(\gamma)}(\bs{\theta})/\LZ_{\GFF}^{(\gamma)}(\bs{\mslitdriv}_t)}
\end{align*}
This gives~\eqref{eqn::transition_density_gff} as desired. Combining~\eqref{eqn::transition_density_gff} and~\eqref{eqn::HL21_asy}, we obtain~\eqref{eqn::transition_density_gff_asy} as desired.
\end{proof}

\subsection{Proof of Proposition~\ref{prop::asymp_forall_U} and proof of Proposition~\ref{prop::asymp}}
\label{subsec::asymp}
\begin{proof}[Proof of Proposition~\ref{prop::asymp_forall_U}]
Assume the same setup as in Lemma~\ref{lem::transition_density_gff}. From~\eqref{eqn::transition_density_gff_asy}, for $t>0$, we have
\begin{align}\label{eqn::asymp_forall_U_lifetime}
\PP_{\gamma}[T>t]\sim \frac{\LZradial(\bs{\theta})}{\LZ_{\GFF}^{(\gamma)}(\bs{\theta})}\exp\left(-(N^2+\gamma^2)Nt\right)\left(1+O(\ee^{-at})\right), 
\end{align}
where the constant in $\sim$ depends on $(N, \gamma)$. Note that the event  $\{\forall j\in \{1, \ldots, N\}, \dist(\eta^j, 0)\le \eps\}$ is closely related to the event $\{T > \frac{|\log \eps|}{2N}\}$. This explains the connection between~\eqref{eqn::asymp_forall_U_lifetime} and~\eqref{eqn::asymp_forall_U}. To make this heuristic argument precise, one may look at~\cite{LawlerRezaeiMinkowskiContent, ZhanGreen2SLE} or~\cite[Proof of Theorem 1.2]{FengWuYangIsing}.
\end{proof}

\begin{proof}[Proof of Proposition~\ref{prop::asymp}]
Let $\varphi$ be a conformal map from $\Omega$ onto $\U$ with $\varphi(u)=0$. We write $\ee^{\ii\bs{\theta}}=\varphi(\bs{x})$ where $\bs{\theta}=(\theta_1, \ldots, \theta_{2N})\in\LX_{2N}$ and $\ee^{\ii\theta_j}=\varphi(x_j)$ for $1\le j\le 2N$. 
From the conformal invariance of the GFF and of the Green's function, we have 
\begin{align*}
\PP_{\gamma}\left[\forall j\in \{1, \ldots, N\}, \dist(\eta^j, u)\le \eps\right]\sim& \PP_{\gamma}\left[\forall j\in \{1, \ldots, N\}, \dist(\varphi(\eta^j), 0)\le |\varphi'(u)|\eps\right]\\
\sim&(|\varphi'(u)|\eps)^{\frac{N^2+\gamma^2}{2}}\LZradial(\bs{\theta})/\LZ_{\GFF}^{(\gamma)}(\bs{\theta});\tag{due to~\eqref{eqn::asymp_forall_U}}\\
\frac{\prod_{j=1}^{2N}\Poisson(\Omega; u, x_j)^{\frac{N}{2}}\exp(-\gamma\Gamma_0(\Omega; \bs{x}; u))}{\CR(\Omega; u)^{\frac{N^2+\gamma^2}{2}}\prod_{\substack{1\le i<j\le 2N\\ i\not\equiv j\mod 2}}\Poisson(\Omega; x_i, x_j)^{\frac{1}{2}}}\sim&|\varphi'(u)|^{\frac{N^2+\gamma^2}{2}}\frac{\exp(-\gamma\Gamma_0(\U; \ee^{\ii\bs{\theta}}; 0))}{\prod_{\substack{1\le i<j\le 2N\\ i\not\equiv j\mod 2}}\Poisson(\U; \ee^{\ii\theta_i}, \ee^{\ii\theta_j})^{\frac{1}{2}}}
\tag{due to~\eqref{eqn::CR_cov}, \eqref{eqn::Poisson_cov}, \eqref{eqn::bPoisson_cov}}\\
\sim& |\varphi'(u)|^{\frac{N^2+\gamma^2}{2}} \exp\left(-\gamma\sum_{j=1}^N(\theta_{2j}-\theta_{2j-1})\right)\prod_{\substack{1\le i<j\le 2N\\ i\not\equiv j\mod 2}}\sin\left(\frac{\theta_j-\theta_i}{2}\right)\\
=&|\varphi'(u)|^{\frac{N^2+\gamma^2}{2}}\LZradial(\bs{\theta})/\LZ_{\GFF}^{(\gamma)}(\bs{\theta}).
\end{align*}
These complete the proof of~\eqref{eqn::asymp_forall}. 
\end{proof}

\begin{remark}
Assume the same notation as in Theorem~\ref{thm::multiplelog_polygon}  with $M=1$. Fix $u_1=u\in\Omega$ and $\gamma_1=\gamma\in\R$.
Using Proposition~\ref{prop::asymp} and cascade relation of multiple $\SLE_4$ and control of partition functions~\cite{PeltolaWuGlobalMultipleSLEs, WuHyperSLE}, one is able to derive the following asymptotic: for $k\in\{1, \ldots, N\}$, 
\begin{align*}
C^{-1} \eps^{\frac{k^2+\gamma^2}{2}}\le \PP_{\gamma}\left[\forall j\in\{1, \ldots, k\}, \dist(\eta^j, u)\le \eps\right]\le C\eps^{\frac{k^2+\gamma^2}{2}}, 
\end{align*}
where $C\in (0,\infty)$ is independent of $\eps$. 
\end{remark}

\end{document}